\documentclass[hidelinks,onefignum,onetabnum]{siamart251104}

\usepackage{lipsum}
\usepackage{amsfonts}
\usepackage{graphicx}
\usepackage{epstopdf}
\ifpdf
  \DeclareGraphicsExtensions{.eps,.pdf,.png,.jpg}
\else
  \DeclareGraphicsExtensions{.eps}
\fi

\newsiamremark{remark}{Remark}
\newsiamremark{hypothesis}{Hypothesis}
\crefname{hypothesis}{Hypothesis}{Hypotheses}
\newsiamthm{claim}{Claim}
\newsiamremark{fact}{Fact}
\crefname{fact}{Fact}{Facts}

\headers{The S\MakeLowercase{i}MPL Method for Multi-Material Topology Optimization}{P. Gangl, B. Keith, D. Kim, B. S. Lazarov, and T. M. Surowiec}

\title{%
  The S\MakeLowercase{i}MPL Method for Density-Based Multi-Material Topology Optimization\thanks{Uploaded on arXiv on \today.
    \funding{DK and BK were supported in part by the U.S.\ Department of Energy, Office of Science Early Career Research Program under Award Number DE-SC0024335 and by the Center for Information Geometric Mechanics and Optimization (CIGMO), a PSAAP-IV Focused Investigatory Center funded by the U.S.\ Department of Energy, National Nuclear Security Administration under Award Number DE-NA0004261.
      BK was also supported in part by the Alfred P.\ Sloan Foundation via a Sloan Research Fellowship in Mathematics.
      The research of TS was supported by the Research Council of Norway, project number 357556 (LaVa).
      The work of PG is partially supported by the State of Upper Austria and by the joint DFG/FWF Collaborative Research Centre CREATOR (DFG: Project-ID 492661287/TRR 361; FWF: Grant-DOI 10.55776/F90) at TU Darmstadt, TU Graz, RICAM and JKU Linz. Part of this work was performed under the auspices of the U.S. Department of Energy by Lawrence Livermore National Laboratory under Contract DE-AC52-07NA27344 and the LLNL-LDRD Program under Project tracking No. 25-ERD-030. Release number LLNL-JRNL-2018905.     }%
  }%
}

\author{
  Peter Gangl\thanks{Johann Radon Institute for Computational and Applied Mathematics, Austrian Academy of Sciences (\email{peter.gangl@ricam.oeaw.ac.at}).}
  \and Brendan Keith\thanks{Division of Applied Mathematics, Brown University, Providence, RI (\email{brendan\_keith@brown.edu}, \email{dohyun\_kim@brown.edu}).}
  \and Dohyun Kim\footnotemark[3]
  \and Boyan S. Lazarov\thanks{Lawrence Livermore National Laboratory (\email{lazarov2@llnl.gov}).}
  \and Thomas M. Surowiec\thanks{Simula Research Laboratory (\email{thomasms@simula.no}).}}

\usepackage{amsopn}

\usepackage{overpic}
\usepackage{graphicx}%
\usepackage{multirow}%
\usepackage{amsmath,amssymb,amsfonts}%
\usepackage{mathrsfs}%
\usepackage[title]{appendix}%
\usepackage{xcolor}%
\usepackage{textcomp}%
\usepackage{manyfoot}%
\usepackage{booktabs}%
\usepackage{algorithm}%
\usepackage{algorithmicx}%
\usepackage{algpseudocode}%
\usepackage{listings}%
\usepackage{bm}
\usepackage{cleveref}

\newcommand{\subto}{\mathop{\mathrm{subject\ to\ }}}
\DeclareMathOperator*{\argmin}{arg\,min}
\DeclareMathOperator{\dd}{d\!}

\usepackage{tcolorbox}

\usepackage{subfig}
\usepackage{tikzscale}
\usepackage{pgfplots}
\pgfplotsset{width=7cm,compat=1.14}
\usetikzlibrary{pgfplots.groupplots}{}

\definecolor{color0}{rgb}{0.7843, 0.7843, 0.7843}
\definecolor{color1}{rgb}{0, 0.4470, 0.7410}
\definecolor{color2}{rgb}{0.8500, 0.3250, 0.0980} \definecolor{color3}{rgb}{0.9290, 0.6940, 0.1250}
\definecolor{color4}{rgb}{0.7060, 0.3840, 0.7650}
\definecolor{color5}{rgb}{0.4660, 0.6740, 0.1880}
\definecolor{color6}{rgb}{0.3010, 0.7450, 0.9330}
\definecolor{color7}{rgb}{0.6350, 0.0780, 0.1840}
\definecolor{color8}{rgb}{0.0, 0.4078, 0.3412}

\newsiamremark{assumption}{Assumption}

\allowdisplaybreaks

\ifpdf
\hypersetup{
  pdftitle={SiMPL method for density-based multi-material topology optimization},
  pdfauthor={P. Gangl, B. Keith, D. Kim, B. S. Lazarov, and T. M. Surowiec}
}
\fi
\begin{document}

\maketitle

% REQUIRED
\begin{abstract}
  We introduce an efficient and scalable method for density-based multi-material topology optimization, integrating classical mirror descent techniques with point-wise polytopal design constraints. Such constraints arise naturally in this class of problems, wherein the vertices of convex polytopes correspond to distinct design states, only one of which should be occupied at each point in space. The framework generates a descending sequence of iterates by penalizing the design space around the previous iterate with a generalized distance function tailored to the convex geometry of the $n$-dimensional polytope. This distance function, called a Bregman divergence, smooths the optimization landscape, ensuring that each iterate strictly satisfies the point-wise constraints. Subsequently, global constraints (e.g., bounds on the structural mass) can be enforced easily by solving a small, finite-dimensional dual problem. The resulting method is simple to implement and demonstrates robustness and efficiency when combined with an Armijo-type line search algorithm. We validate the method in structural design problems involving the optimal arrangement of both isotropic and anisotropic materials, as well as magnetic flux optimization in electric motors.
\end{abstract}
\begin{keywords}
  Multi-Material Topology Optimization, Mirror Descent, Numerical Optimization, Line Search Algorithm, Engineering Design
\end{keywords}

% REQUIRED
\begin{MSCcodes}
  74P05, 74P15, 90C30, 65K10
\end{MSCcodes}


\maketitle

%!TEX root = ./main.tex
\section{Introduction}

A wide range of engineering design problems involves finding the optimal arrangement of multiple materials. Examples can be found in \cite{BendsoeSigmund2004}, which investigates the optimal placement of materials with different stiffnesses in structural compliance minimization, and \cite{Cherriere}, which considers the distribution of current phases in electric motor design, among many other recent studies. Although incorporating multiple materials introduces additional modeling and computational challenges, most established approaches for classical (single-material) topology optimization have been extended to this setting, including level-set methods based on shape sensitivities \cite{AllaireDapognyEtAl2014, Laurain2025189229} or topological derivatives \cite{Gangl2020, NodaYamada2022}, phase-field formulations \cite{Adam2018, Adam2019, Blank2025}, and traditional material interpolation schemes \cite{BendsoeSigmund2004}.

Level-set methods offer the advantage of producing crisp material distributions without intermediate states, but density-based approaches remain the predominant choice in industrial applications because of their ease of implementation and compatibility with standard optimization software. In multi-material density-based topology optimization, one typically introduces a polygonal or polyhedral interpolation domain whose vertices correspond to distinct materials. A straightforward construction uses orthogonal domains such as unit squares or cubes, yet this restricts the admissible number of materials to powers of two. Attempts to bypass this limitation by duplicating materials can introduce undesirable biases in the optimized design.
Embedding material states into a higher-dimensional simplex alleviates this issue \cite{Adam2018, Blank2025}, but in some applications it is advantageous to encode physical intuition directly into the geometry of the domain, leading to a general (convex) polygonal or polyhedral interpolation domains \cite{Cherriere, CherriereGanglKrenn2025}. As observed in \cite{Cherriere}, a central difficulty for density-based multi-material optimization over these domains is the complexity of the associated orthogonal projection. While efficient projection algorithms exist for simplices of arbitrary dimension, projection onto general high-dimensional convex polytopes is more computationally involved.
Moreover, such projections are mainly suited to lowest-order finite element discretizations of the density field, as they need only be applied at element midpoints or nodes to enforce the design constraints.
This inhibits the development of higher-order discretizations for topology optimization, which could have superior stability properties; cf.\ \cite[Experiment~2]{simpl-math}.

To address these challenges, we introduce the Simplicial Mirror descent with a Projected Latent variable (SiMPL) method for multi-material topology optimization, which generalizes the Sigmoidal Mirror descent with a Projected Latent variable method \cite{simpl-engrg, simpl-math}, originally developed for single-material problems. The SiMPL method replaces explicit pointwise projections onto the design space with a mirror map induced by a Legendre function whose effective domain matches the target $n$-dimensional convex polytope. In contrast to the original, single-material SiMPL method, which employs the Fermi–Dirac binary entropy to enforce pointwise bound constraints, we construct a new Legendre function based on the Gibbs entropy for simplices, but tailored to general convex polytopes. This allows the design to evolve freely through an \textit{unconstrained} latent variable $\bm{\psi}$, while the design variable $\bm{\eta} = \bm{\eta}(\bm{\psi})$ automatically satisfies the problem's pointwise polytopal design constraints. The resulting algorithm consists of two stages: an {unconstrained} mirror-descent update on the latent variable, followed by a projection that enforces global constraints such as bounds on the total mass. The latter reduces to a finite-dimensional dual problem whose size equals the number of global constraints. The proposed method preserves the main advantages of its single-material counterpart \cite{simpl-engrg, simpl-math}, namely conceptual simplicity, ease of implementation, and fast convergence with standard line-search strategies requiring minimal parameter tuning.

\section{Multi-material Topology Optimization}
We begin by stating an abstract topology optimization problem over a design domain $\Omega \subset \mathbb{R}^d$. Hereafter, we let $\bm{u}$ denote the state variable, which implicitly depends on the design variable $\bm{\eta}\in L^\infty(\Omega;\mathbb{R}^n)$ via the abstract equation
\begin{equation*}
  \mathcal{L}(\bm{\eta},\bm{u})=0.
\end{equation*}
Our objective is to minimize a functional $F: L^\infty(\Omega; \mathbb{R}^n) \to \mathbb{R}$, which incorporates the state variable $\bm{u} = \bm{u}(\bm{\eta})$, subject to both local and global constraints on the design variable.
To this end, let $0 < M_0 < \infty$ and $0 < M_i \leq \infty$ for each $i=1,...,n$.
The density-based multi-material topology optimization problem reads as follows:
\begin{subequations}\label{eq:prob-min}
  \begin{align}
    \label{eq:min-obj}
    \min_{\bm{\eta}\in L^\infty(\Omega;\mathbb{R}^n)}
     &
    F(\bm{\eta})                                                        \\
    \subto
    \label{eq:min-sum1}
     & \sum_{i=1}^n\eta_i=1,                                            \\
    \label{eq:min-pos}
     & \eta_i\geq 0\text{ for all }i=1,...,n,                           \\
    \label{eq:min-mass-all}
     & \sum_{i=1}^n\int_\Omega \rho_i\eta_i\dd x\leq M_0,               \\
    \label{eq:min-mass-ind}
     & \int_\Omega \rho_i\eta_i\dd x\leq M_i\text{ for each }i=1,...,n.
  \end{align}
\end{subequations}
In this set-up, $\bm{\eta}$ plays the role of an indicator field with $\bm{\eta}(x)=\bm{e}_i$ implying that $x$ is fully occupied by the $i$-th material.
In particular, the local (pointwise) constraints \eqref{eq:min-sum1} and \eqref{eq:min-pos} together enforce a volume-filling (partition of unity) property on the design variable $\bm{\eta}$, while the global constraints \eqref{eq:min-mass-all} and \eqref{eq:min-mass-ind} enforce total and individual mass requirements for the material states, respectively.

\subsection{Polytopal constraints}

Note that the constraints \eqref{eq:min-sum1}--\eqref{eq:min-pos} ensure that $\bm{\eta}(x)$ is contained in the standard $(n-1)$-simplex at every point $x$ in $\Omega$.
More generally, for a convex polytope $P \subset \mathbb R^n$, a matrix $W \in \mathbb R^{r \times n}$, and a vector $\bm{b} \in \mathbb R^r$, we may consider
\begin{subequations}\label{eq:admin}
  \begin{align}
    \mathcal{K} & :=\{\bm{\eta}\in L^\infty(\Omega;\mathbb{R}^n)\mid\bm{\eta}(x)\in P\text{ a.e.\ }x\in\Omega\},\label{eq:admin_a}       \\
    \mathcal{M} & :=\{\bm{\eta}\in L^\infty(\Omega;\mathbb{R}^n)\mid\int_\Omega W\bm{\eta}\dd x\preccurlyeq \bm{b}\}, \label{eq:admin_b}
  \end{align}
\end{subequations}
where $\bm{a}\preccurlyeq\bm{b}$ (resp.\ $\bm{a}\prec\bm{b}$) denotes the element-wise inequality, meaning $a_i\leq b_i$ (resp.\ $a_i<b_i$) for each index $i$.
In this more general setting, $\mathcal{K}$ encodes the pointwise polytopal constraint, while $\mathcal{M}$ indicates $r$ global linear constraints.
Likewise, the feasible set $\mathcal{A}$ is the intersection of these two sets, $\mathcal{A}=\mathcal{K}\cap \mathcal{M}$.
In this case, we also define subsets of $\mathcal{K}$ and $\mathcal{A}$ consisting of the member functions that are uniformly bounded away from $\partial P$; namely,
\begin{equation}\label{eq:pointwise-int}
  \mathring{\mathcal{K}}:=\Big\{\bm{\eta}\in \mathcal{K}\mid \operatorname*{ess\,inf}_{x\in\Omega}\big(\operatorname{dist}(\bm{\eta}(x),\partial P)\big)>0\Big\},\qquad\mathring{\mathcal{A}}:=\mathring{\mathcal{K}}\cap\mathcal{M},
\end{equation}
where $\operatorname{dist}(\bm{p},\partial P):=\inf_{\bm{y}\in \partial P}|\bm{p}-\bm{y}|$ is the Euclidean distance from the point $\bm{p}$ to the boundary $\partial P$.
As such, each member of $\mathring{\mathcal{K}}$ and $\mathring{\mathcal{A}}$ is uniformly bounded away from $\partial P$.

In general, we consider multi-material topology optimization problems of the form~
\begin{equation}
  \min_{\bm{\eta}\in \mathcal{A}}F(\bm{\eta}).
  \label{eq:GeneralMMTO}
\end{equation}
Using the above notation, \eqref{eq:prob-min} can be seen as special case of~\eqref{eq:GeneralMMTO} by setting
\begin{gather*}
  P:=\Delta^{n-1}=\big\{\bm{\eta}\in \mathbb{R}^n:\sum_{i=1}^n\eta_i=1,\;\eta_i\geq 0\text{ for all }i=1,...,n\big\},\\
  W:=\begin{bmatrix}\bm{\rho}^\top\\\operatorname{diag}(\bm{\rho})\end{bmatrix},\qquad \bm{b}:=\begin{bmatrix}M_0\\\vdots\\M_n\end{bmatrix}.
\end{gather*}

\subsection{Differentiability}

In density-based topology optimization, auxiliary fields, such as filtered densities, are typically introduced to ensure well-posedness and manufacturability~\cite{lazarov2011-filter,bendsoe1988,Amir2012-robust}. For brevity, we subsume these operations within the definition of the objective functional $F$. Thus, while the Fr\'echet derivative $F'(\bm{\eta})$ generally resides in the dual space $(L^\infty(\Omega; \mathbb{R}^n))'$, many topology optimization problems possess enough structure to admit a primal representation of the derivative in $L^\infty(\Omega; \mathbb{R}^n)$ \cite{simpl-math}. We formalize this property via the following assumption.
\begin{assumption}[Differentiability]\label{asum:grad}
  The objective $F$ is continuously Fr\'echet differentiable on $\mathcal{A}$. In addition, its derivative admits a primal representation $\nabla F \in L^\infty(\Omega;\mathbb{R}^n)$ characterized by the variational equation
  \begin{equation*}
    \int_\Omega \nabla F(\bm{\eta}) \cdot \bm{q} \dd x = \langle F'(\bm{\eta}), \bm{q} \rangle \quad\text{for all }\bm{q} \in L^\infty(\Omega;\mathbb{R}^n).
  \end{equation*}
\end{assumption}

\subsection{Notation}
\label{sub:Notation}
Except in \Cref{lem:parametrization}, below, we perform the minor crime of using the same symbol to denote a function with arguments in $\mathbb{R}^n$ and its superposition operator, which takes arguments in $L^\infty(\Omega;\mathbb{R}^n)$. Throughout, $|\bm{y}|$ denotes the Euclidean norm of $\bm{y}\in\mathbb{R}^n$. Bold symbols denote vectors or vector-valued functions, and a plain symbol with a subscript denotes the corresponding component, e.g., $y_i$ denotes the $i$-th component of $\bm{y}$.

\section{Bregman Divergence and Mirror Descent}
In this section, we motivate the general optimization framework that leads to the SiMPL method, as presented in \Cref{alg:pmd-gbb}, below.
To maintain focus on the conceptual transition, we defer certain technical details to subsequent sections and proceed with a formal derivation of the algorithm's main features.

\subsection{Mirror Descent Method}
Consider a first-order approximation of the objective function $F$ around the previous iterate $\bm{\eta}^k$:
\begin{equation}\label{eq:first-order-approx}
  J(\bm{\eta};\bm{\eta}^k)=F(\bm{\eta}^k)+\int_\Omega \nabla F(\bm{\eta}^k)\cdot(\bm{\eta}-\bm{\eta}^k)\dd x.
\end{equation}
The well-known \textit{projected steepest descent method} solves an $L^2$-regularized subproblem based on this first-order approximation.
We note that \cref{asum:grad} implies that $F$ and $\nabla F$ are well-defined for all $L^p(\Omega;\mathbb{R}^n)$ with $1\leq p\leq \infty$.
By embedding the problem in the $L^2$-topology, we consider:
\begin{equation*}
  \min_{\bm{\eta} \in \mathcal{A}} J(\bm{\eta};\bm{\eta}^k) + \frac{1}{2\alpha_k} \|\bm{\eta} - \bm{\eta}^k\|_{L^2(\Omega)}^2.
\end{equation*}
The unique solution to this constrained minimization problem is characterized by
\begin{equation}
  \label{eq:PGD}
  \bm{\eta}^{k+1} = \operatorname{Proj}_{ L^2} \Big( \bm{\eta}^k - \alpha_k \nabla F(\bm{\eta}^k) \Big),
\end{equation}
where $\operatorname{Proj}_{ L^2}$ denotes the canonical $L^2$-projection onto the admissible set $\mathcal{A}$.

Mirror descent methods, on the other hand, replace the $L^2$-regularization term with a more general distance-like function:
\begin{equation*}
  \min_{\bm{\eta}\in\mathcal{A}}J(\bm{\eta};\bm{\eta}^k)+\frac{1}{\alpha_k}\int_\Omega \mathcal{D}_R(\bm{\eta},\bm{\eta}^k)\dd x.
\end{equation*}
Here, $\mathcal{D}_R$ is the Bregman divergence generated by a strictly convex function $R:\mathbb{R}^n\rightarrow \mathbb{R}\cup \{+\infty\}$,
\begin{equation*}
  \mathcal{D}_R(\bm{\eta},\bm{\eta}^k)=R(\bm{\eta})-R(\bm{\eta}^k)-\nabla R(\bm{\eta}^k)\cdot(\bm{\eta}-\bm{\eta}_k).
\end{equation*}
When $R(\bm{\eta})=|\bm{\eta}|^2/2$, we recover the projected steepest descent method~\eqref{eq:PGD}.
When $R$ is only subdifferentiable, then the gradient $\nabla R(\bm{\eta}^k)$ can be replaced with any subgradient of $R$ at $\bm{\eta}^k$.
Similar to the steepest descent method, we can express the next iterate by
\begin{equation}
  \label{eq:md-full-formula}
  \bm{\eta}^{k+1}=\operatorname{Proj}_{R}\Big((\nabla R^*)\big(\nabla R(\bm{\eta}^k)-\alpha_k\nabla F(\bm{\eta}^k)\big)\Big).
\end{equation}
Here, $\nabla R^*$ is the gradient of the Fenchel conjugate of $R$ and $\operatorname{Proj}_{R}:\mathring{\mathcal{K}}\rightarrow \mathcal{A}$ is the Bregman projection
\begin{equation*}
  \operatorname{Proj}_{R}(\bm{q}):=\argmin_{\bm{\eta}\in\mathcal{A}}\int_\Omega \mathcal{D}_R(\bm{\eta},\bm{q})\dd x.
\end{equation*}

\begin{algorithm}[t]
  \caption{The SiMPL Method for Multi-Material Topology Optimization}\label{alg:pmd-gbb}
  \algrenewcommand\algorithmicindent{0.75em}%
  \begin{algorithmic}[1]
    \Require initial design $\bm{\eta}^0\in\mathring{\mathcal{K}}$, exit tolerance $\mathtt{tol} > 0$, and $c_1>0$.
    % \State $k \gets -1$
    \State $\bm{\psi}^0 \gets \nabla R(\bm{\eta}^0)$
    \For{$k=0,1,2,\ldots$}
    \State Compute $\displaystyle\alpha_k > 0$ \Comment{Eq.~\eqref{eq:gbb}}
    \While{\texttt{true}}
    \State$\bm{\psi}^{k+1/2} \gets \bm{\psi}^k - \alpha_k \nabla F(\bm\eta^k)$
    \State$\bm{\psi}^{k+1}=\bm{\psi}^{k+1/2}-W^T\bm{\mu}^{k+1}$ \Comment{Eq.~\eqref{eq:dual-problem}}
    \State$\bm{\eta}^{k+1} \gets \nabla R^*(\bm{\psi}^{k+1})$ \Comment{Eq.~\eqref{eq:polymap}}
    \If{the Armijo condition \eqref{eq:armijo} is satisfied}
    \State \textbf{break}
    \EndIf
    \State $\alpha_k \gets \alpha_k/2$
    \EndWhile
    \If{$\texttt{res}_{k} < \mathtt{tol}$} \Comment{Eq.~\eqref{eq:KKT_estimator}}
    \State \textbf{break}
    \EndIf
    \EndFor

    \State \Return $\bm{\eta}^{k+1},\;F(\bm{\eta}^{k+1})$
  \end{algorithmic}
\end{algorithm}

\begin{figure*}
  \centering
  \begin{tabular}{cc}
    \includegraphics[width=0.35\textwidth]{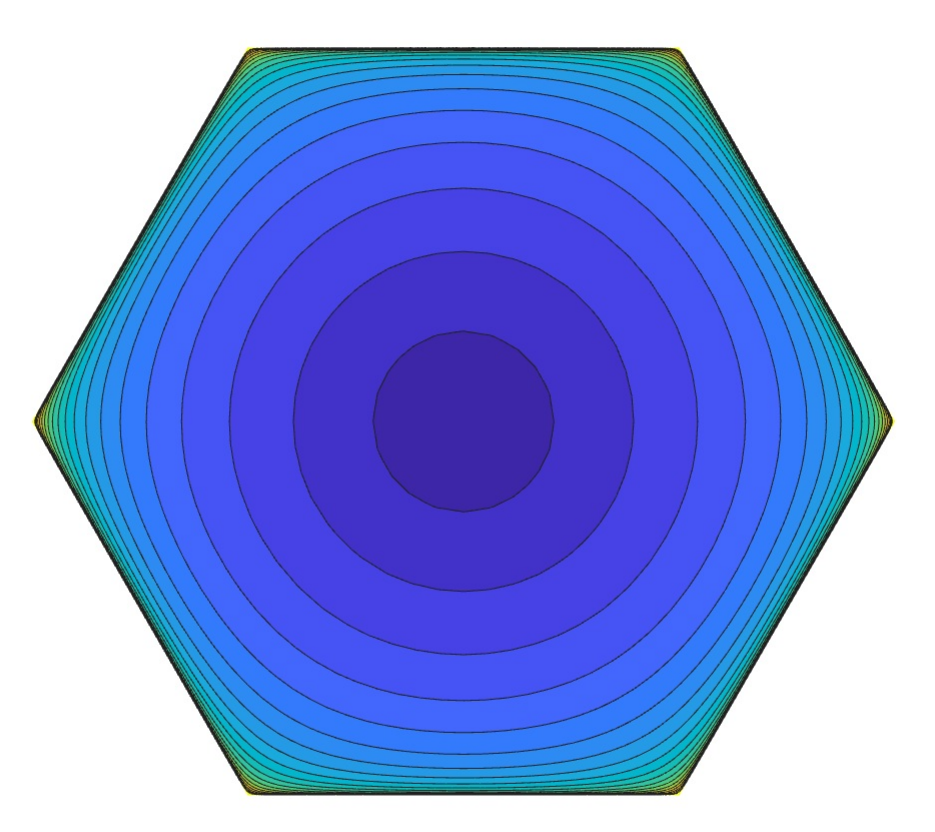} &
    \includegraphics[width=0.42\textwidth]{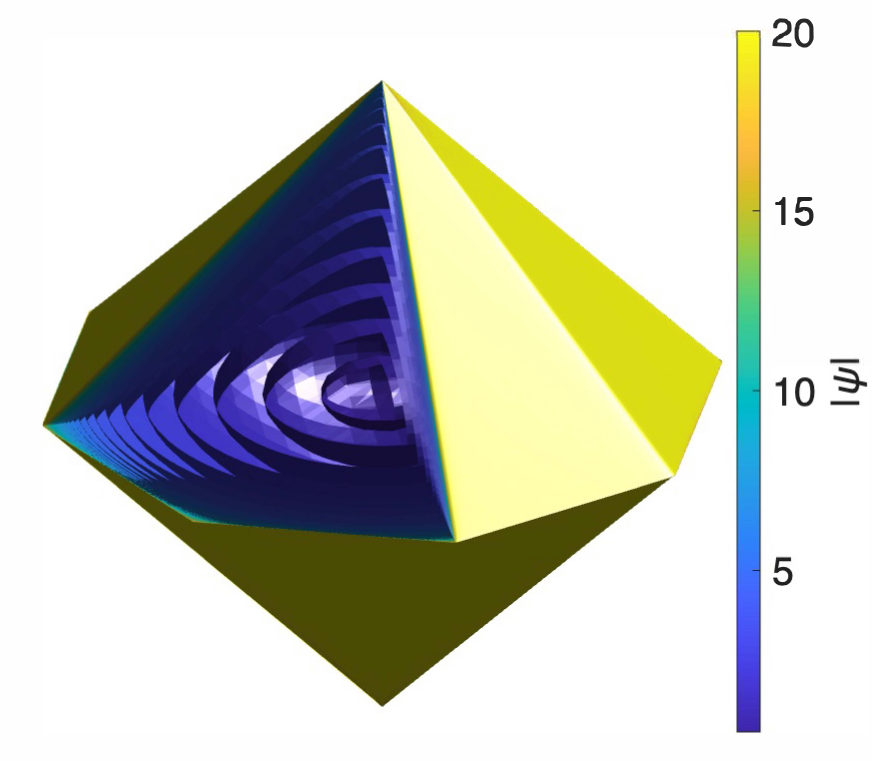}     \\
    (a)                                                      & (b)
  \end{tabular}
  \caption{Polytopes parametrized using $\nabla R^*(\bm\psi)$; cf.~\Cref{thm:polylegendre}. (a) A hexagon and (b) a bipyramid with an octagonal base. Colors indicate the magnitude of the latent variable $|\bm{\psi}|$.\label{fig:mapping}}
\end{figure*}

We aim to choose $R:\mathbb{R}^n\rightarrow \mathbb{R}\cup \{+\infty\}$ whose effective domain,
\begin{equation*}
  \operatorname{dom}(R):=\{\bm{p}\in\mathbb{R}^n:R(\bm{p})<+\infty\},
\end{equation*}
coincides with the polytope $P$, and whose gradient is (left-)invertible in $\operatorname{int}(P)$ but diverges on $\partial P$.
More specifically, we will choose $R$ of Legendre type (cf.~\Cref{def:legendre} and \Cref{thm:polylegendre}, below) so that $\nabla R^*$ is a smooth mapping $\nabla R^*\colon \mathbb{R}^n \to \operatorname{int}(P)$ as depicted in \Cref{fig:mapping}.
In fact, $\nabla R^*$ is the inverse of $\nabla R$, implying that
\begin{equation*}
  \nabla R^*(\nabla R(\bm{\eta}))=\bm{\eta}\quad\text{for all }\bm{\eta}\in\mathring{\mathcal{K}}.
\end{equation*}
This relationship allows us to define a latent variable $\bm{\psi}\in L^\infty(\Omega;\mathbb{R}^n)$ parameterizing the design variable via $\bm{\eta}=\nabla R^*(\bm{\psi})$.
In turn, the update rule~\cref{eq:md-full-formula} becomes
\begin{subequations}\label{eq:md-simpl}
  \begin{align}
    \bm{\psi}^{k+1/2} & = \bm{\psi}^k - \alpha_k \nabla F(\bm{\eta}^k), \\
    \bm{\psi}^{k+1}   & = \bm{\psi}^{k+1/2}-W^\top\bm{\mu}^k,
  \end{align}
\end{subequations}
where $\bm{\mu}^k\in\mathbb{R}^m_{\geq0}$ is chosen so that
\begin{equation}
  \label{eq:update_with_proj}
  \bm{\eta}^{k+1}:=(\nabla R)^*(\bm{\psi}^{k+1/2}-W^\top\bm{\mu}^k)=\argmin_{\bm{\eta}\in\mathcal{A}}\int_\Omega \mathcal{D}_R(\bm{\eta},\bm{\eta}^{k+1/2})\dd x.
\end{equation}

Mirror descent methods are known to be convergent for convex optimization problems \cite{Beck2003} under suitable assumptions. In \cite{simpl-math}, the authors proved that, when combined with backtracking line search, the SiMPL method for single-material topology optimization yields monotone-decreasing objectives and demonstrated its convergence.
\subsection{Choice of $R$}
The choice of the convex function $R$ plays a central role in mirror descent methods.
To encode the pointwise polytopal constraint, we make use of the concept of Legendre functions \cite[Section 26]{rockafellar1970convex}.
The following definitions and results are well-known in the literature on convex optimization \cite{rockafellar1970convex} but relevant to the construction of the SiMPL method.
\begin{definition}[Effective Domain and Subdifferential]
  Let $f:\mathbb{R}^n\to\mathbb{R}\cup\{+\infty\}$ be a proper convex function. We denote the subdifferential of $f$ at $\bm{x}\in\mathbb{R}^n$ by the (possibly empty) set $\partial f(\bm{x})$ defined by
  \begin{equation*}
    \partial f(\bm{x}):=\{\bm{x}^*\in\mathbb{R}^n\mid \bm{x}^*\cdot(\bm{y}-\bm{x})+f(\bm{x})\leq f(\bm{y})\;\forall \bm{y}\in\mathbb{R}^n\}.
  \end{equation*}
  The effective domain of $f$ (resp. $\partial f$) is denoted by
  \begin{equation*}
    \operatorname{dom}(f):=\{\bm{x}\in\mathbb{R}^n\mid f(\bm{x})<+\infty\},\quad \operatorname{dom}(\partial f):=\{\bm{x}\in\mathbb{R}^n\mid \partial f(\bm{x})\neq \emptyset\}.
  \end{equation*}
  When $\partial f(\bm{x})$ is a singleton, $f$ is differentiable at $\bm{x}$, and we write $\nabla f(\bm{x})$ for its unique element.
\end{definition}
\begin{definition}[Legendre Function]\label{def:legendre}
  Let $R:\mathbb{R}^n\to\mathbb{R}\cup\{+\infty\}$ be a proper convex function. We say $R$ is
  \begin{enumerate}
    \item \underline{essentially strictly convex}, if $R$ is strictly convex on every convex subset of $\operatorname{dom}(\partial R)$,
    \item \underline{essentially smooth}, if $R$ is differentiable on $\operatorname{int}\operatorname{dom}(R)\neq\emptyset$, and $|\nabla R(\bm{x}_n)|\to+\infty$ whenever $\bm{x}_n\in \operatorname{dom}(R)$ and $\bm{x}_n\to\partial \operatorname{dom}(R)$.
  \end{enumerate}
  If $R$ is both essentially strictly convex and essentially smooth, then $R$ is said to be of \underline{Legendre type}.
\end{definition}
\begin{lemma}[Legendre Function \cite{rockafellar1970convex}]\label{lem:legendre}
  Let $R:\mathbb{R}^n\to\mathbb{R}\cup\{+\infty\}$ be of Legendre type. Then its convex conjugate $R^*$, defined by
  \begin{equation*}
    R^*(\bm{y}):=\sup_{\bm{x}\in\mathbb{R}^n}\{\bm{x}\cdot \bm{y} - R(\bm{x})\mid \bm{x}\in \mathbb{R}^n\},
  \end{equation*}
  is also of Legendre type.
  The gradient mapping $\nabla R:\operatorname{dom}(\partial R)\to \operatorname{dom}(\partial R^*)$ is a homeomorphism with $\nabla R=(\nabla R^*)^{-1}$.
  Furthermore, $\operatorname{dom}(\partial R^*)=\mathbb{R}^n$ if and only if $R$ is superlinear, meaning $R(\bm{x})/|\bm{x}|\to+\infty$ as $|\bm{x}|\to+\infty$.
\end{lemma}
This lemma implies that a superlinear Legendre function $R$ induces a natural parametrization of $\operatorname{dom}(R)$ by $\nabla R^*$, which is a smooth mapping from $\mathbb{R}^n$ to $\operatorname{int}(\operatorname{dom}(R))$.

We are particularly interested in the case with $\operatorname{dom}(R)=P$, where $P$ is a convex polytope.
In \Cref{thm:polylegendre}, we construct a superlinear Legendre function $R$ with $\operatorname{dom}(R)=P$.
To this end, we first consider the special case where $P$ is the standard ($n-1$)-simplex $\Delta^{n-1}$, and then extend the result to general polytopes by using the simplex as a building block.
\begin{lemma}[Restricted Gibbs Entropy]\label{lem:gibbs}
  The affine hull of the standard $(n-1)$-simplex $\Delta^{n-1}:=\operatorname{conv}\{\bm{e}_1,\cdots,\bm{e}_n\}$ is given by
  \begin{equation*}
    \operatorname{aff}(\Delta^{n-1}):=\Big\{\sum_{i=1}^n \alpha_i\bm{e}_i\mid \bm{\alpha}\in\mathbb{R}^n, \sum_{i=1}^n\alpha_i=1\Big\}.
  \end{equation*}
  We consider the restricted topology on $\operatorname{aff}(\Delta^{n-1})$ induced by the ambient topology of $\mathbb{R}^n$, i.e., the collection of open sets in $\operatorname{aff}(\Delta^{n-1})$ is given by
  \begin{equation*}
    \tau:=\{U\cap \operatorname{aff}(\Delta^{n-1})\mid U\text{ is open in }\mathbb{R}^n\}.
  \end{equation*}
  Let $R_n:\mathbb{R}^n\to\mathbb{R}\cup\{+\infty\}$ be defined by
  \begin{equation*}
    R_n(\bm{x})=\begin{cases}
      \sum_{i=1}^{n} x_i\ln x_i & \text{when }\bm{x}\in\Delta^{n-1}, \\
      +\infty                   & \text{otherwise,}
    \end{cases}
  \end{equation*}
  with the convention $0\ln 0=0$.
  Then $R_n$ is a superlinear Legendre function in the restricted topology.
  In addition, the gradient of the convex conjugate $R_n^*$ acts as a left inverse of the subdifferential $\partial R_n$ in the sense that
  \begin{equation*}
    \begin{aligned}
      \nabla R_n^*(\partial R_n(\bm{x})) & =\{\bm{x}\}\text{ for all }\bm{0}\prec\bm{x}\in\Delta^{n-1}.
    \end{aligned}
  \end{equation*}
  Here, $\nabla R_n^*(\bm{y})$ is the softmax function
  \begin{equation*}
    \nabla R_n^*(\bm{y})=\frac{\exp(\bm{y})}{\sum_{i=1}^n\exp(y_i)}=:\operatorname{softmax}(\bm{y})\text{ for all }\bm{y}\in\mathbb{R}^n.
  \end{equation*}
  Here, $\exp(\bm{y})_j:=\exp(y_j)$ for all $j=1, \ldots,n$.
\end{lemma}
\begin{proof}
  This function has been well-studied in the literature on convex optimization, see in particular \cite{Beck2003,teboulle2018nolip}.
\end{proof}
\begin{remark}\label{rmk:simplex}
  The function $R_n$ is not differentiable in the ambient topology of $\mathbb{R}^n$ since its effective domain $\Delta^{n-1}$ has an empty interior. However, it is subdifferentiable with its subdifferential given by
  \begin{equation*}
    \partial R_n(\bm{x})=\begin{cases}
      \{\ln(\bm{x})+\alpha\bm{1}\mid \alpha\in\mathbb{R}\} & \text{when }\bm{0}\prec\bm{x}\in\Delta^{n-1}, \\
      \emptyset                                            & \text{otherwise,}
    \end{cases}
  \end{equation*}
  where $\ln(\bm{x})_j:=\ln(x_j)$ for all $j=1,\ldots,n$.
  Notice that $\nabla R_n^*$ is translation invariant, i.e., $\nabla R_n^*(\bm{y}+\alpha\bm{1})=\nabla R_n^*(\bm{y})$ for all $\alpha\in\mathbb{R}$.
  This implies that any two subgradients of $R_n$ at $\bm{x}\in\Delta^{n-1}$ are mapped to the same point $\bm{x}$ by $\nabla R_n^*$.
  This can also be utilized numerically to avoid a numerical overflow issue when computing $\nabla R_n^*(\bm{y})$ for large $\bm{y}$ by
  \begin{equation*}
    \nabla R_n^*(\bm{y})=\frac{\exp(\bm{y}-\max_i y_i)}{\sum_{i=1}^n\exp(y_i-\max_i y_i)}.
  \end{equation*}
\end{remark}

\begin{theorem}[Polytopal Legendre Function]\label{thm:polylegendre}
  For $1\leq n < q$, let $\bm{v}_1,\cdots,\bm{v}_q\in\mathbb{R}^n$ be the vertices of a convex polytope $P=\operatorname{conv}\{\bm{v}_1,\cdots,\bm{v}_q\}$ with non-empty interior. Then the following function $R:\mathbb{R}^n\to\mathbb{R}\cup\{+\infty\}$ is a superlinear Legendre function with $\operatorname{dom}(R)=P$:
  \begin{equation}\label{eq:polylegendre}
    R(\bm{x})=\inf_{\bm{\lambda}\in\Delta^{q-1}}\left\{\sum_{i=1}^q\lambda_i\ln \lambda_i=R_q(\bm{\lambda})\mid V\bm{\lambda}=\bm{x}\right\}.
  \end{equation}
  Here, $V\in\mathbb{R}^{n\times q}$ is the vertex matrix of $P$, whose $i$-th column is $\bm{v}_i$.
  Its convex conjugate is given by
  \begin{equation}\label{eq:polylegendre-conjugate}
    R^*(\bm{y})=R^*_q(V^T\bm{y}),
  \end{equation}
  with the gradient mapping
  \begin{equation}\label{eq:polymap}
    \nabla R^*:\bm{y}\in\mathbb{R}^n\mapsto V\nabla R_q^*(V^T\bm{y})=V\operatorname{softmax}(V^T\bm{y})\in \operatorname{int}(P).
  \end{equation}
  Finally, the $\nabla R^*$ is homeomorphism between $\mathbb{R}^n$ and $\operatorname{int}(P)$ with inverse $(\nabla R^*)^{-1}=\nabla R$.
\end{theorem}
\begin{proof}
  Translating the polytope $P$ if necessary, we assume that the origin $\bm{0}$ is the vertex centroid, i.e., $V\bm{1}=\bm{0}$.
  Since $P$ has non-empty interior, there exists a simplex $S=\operatorname{conv}\{\bm{v}_{i_1},\cdots,\bm{v}_{i_{n+1}}\}\subset P$ with non-empty interior, which implies that $V$ has full row rank.
  Therefore, $V:\mathbb{R}^q\to\mathbb{R}^n$ (resp.\ $V^T:\mathbb{R}^n\to\mathbb{R}^q$) is surjective (resp.\ injective) with $\operatorname{span}\{\bm{1}\}\subset \operatorname{ker}(V)$.
  Observe that
  \begin{align*}
    R^*(\bm{y}) & =\sup_{\bm{x}\in\mathbb{R}^n}\{\bm{x}\cdot\bm{y}-R(\bm{x})\}                                                                                              \\
                & =\sup_{\bm{x}\in\mathbb{R}^n}\left\{\bm{x}\cdot\bm{y}-\inf_{\bm{\lambda}\in\Delta^{q-1}}\left\{R_q(\bm{\lambda})\mid V\bm{\lambda}=\bm{x}\right\}\right\} \\
                & =\sup_{\bm{\lambda}\in\Delta^{q-1}}\left\{\bm{\lambda}\cdot V^T\bm{y}-R_q(\bm{\lambda})\right\}=R_q^*(V^T\bm{y}),
  \end{align*}
  where the last line follows from the surjectivity of $V:\mathbb{R}^q\to\mathbb{R}^n$ and the definition of $R_q^*$.

  Following \Cref{lem:legendre}, $R$ is of Legendre type and superlinear if and only if $R^*$ is of Legendre type with $\operatorname{dom}(\partial R^*)=\mathbb{R}^n$.
  By the chain rule, $R^*$ is differentiable everywhere with $\nabla R^*(\bm{y})=V\nabla R_q^*(V^T\bm{y})$ for all $\bm{y}\in\mathbb{R}^n$.
  Therefore, it is enough to show that $R^*$ is essentially strictly convex.

  For any $\bm{z}\in\mathbb{R}^q$, the Hessian of $R_q^*$ at $\bm{z}$ is
  \begin{equation*}
    \nabla^2 R_q^*(\bm{z})=\operatorname{diag}(\bm{p})-\bm{p}\bm{p}^T,\qquad \bm{p}=\frac{\exp(\bm{z})}{\sum_j \exp(z_j)},
  \end{equation*}
  so that for any $\bm{u}\in\mathbb{R}^q$,
  \begin{equation*}
    \bm{u}^T\nabla^2 R_q^*(\bm{z})\bm{u}=\sum_{i=1}^q p_i u_i^2-\Big(\sum_{i=1}^q p_i u_i\Big)^2=\operatorname{Var}_{\bm{p}}(\bm{u}).
  \end{equation*}
  This variance is non-negative and vanishes if and only if $\bm{u}=\alpha\bm{1}$ for some $\alpha\in\mathbb{R}$. Hence $\nabla^2 R_q^*(\bm{z})$ is positive semidefinite with one-dimensional kernel $\operatorname{span}\{\bm{1}\}$.

  By the chain rule, the Hessian of $R^*$ at any $\bm{y}\in\mathbb{R}^n$, given by
  \begin{equation*}
    \nabla^2 R^*(\bm{y})=V\,\nabla^2 R_q^*(V^T\bm{y})\,V^T,
  \end{equation*}
  is also positive semidefinite since $\nabla^2 R_q^*(V^T\bm{y})$ is positive semidefinite.
  Moreover, the kernel of $\nabla^2 R^*(\bm{y})$ is characterized by
  \begin{equation*}
    \operatorname{ker}(\nabla^2 R^*(\bm{y}))=\{\bm{u}\in\mathbb{R}^n\mid V^T\bm{u}\in\operatorname{ker}(\nabla^2 R_q^*(V^T\bm{y}))\}=\operatorname{range}(V^T)\cap \operatorname{span}(\bm{1}).
  \end{equation*}
  Now, $V\bm{1}=\bm{0}$ gives $\bm{1}\in\ker(V)=\operatorname{range}(V^T)^\perp$. Therefore, $\nabla^2 R^*(\bm{y})$ has a trivial kernel and $R^*$ is strictly convex.
\end{proof}

The Legendre function $R$ constructed in \Cref{thm:polylegendre} lifts the Gibbs entropy $R_q$ of \Cref{lem:gibbs} from the simplex $\Delta^{q-1}$ to the polytope $P$ via the linear map $V$. While $R$ itself is defined implicitly through the variational problem \eqref{eq:polylegendre}, both its convex conjugate $R^*$ and the gradient mapping $\nabla R^*$ admit closed-form expressions \eqref{eq:polylegendre-conjugate}--\eqref{eq:polymap}.
\begin{remark}
  Mirror descent only requires evaluation of $\nabla R^*$, which is given pointwise by
  \begin{equation*}
    \bm{\eta}(x)=\nabla R^*\big(\bm{\psi}(x)\big)=V\operatorname{softmax}\big(V^T\bm{\psi}(x)\big)\quad\text{for a.e.\ }x\in\Omega.
  \end{equation*}
  The intermediate vector $\bm{\lambda}(x)=\operatorname{softmax}\big(V^T\bm{\psi}(x)\big)\in\Delta^{q-1}$ furnishes a (maximum-entropy) set of barycentric coordinates of $\bm{\eta}(x)$ with respect to the vertices of $P$, closely related to the maximum-entropy coordinates of \cite{sukumar2004}.
\end{remark}

\section{Projected Mirror Descent with Latent Variable}
Building on the Legendre function of the previous section, this section establishes well-posedness of the mirror descent update \eqref{eq:md-full-formula} in the infinite-dimensional setting.
We begin by showing that the essential interior $\mathring{\mathcal{K}}$ of the constraint set $\mathcal{K}$ can be parametrized by a latent variable $\bm{\psi}$ using the homeomorphism property $\nabla R^* =(\nabla R)^{-1}$.
In the remainder of the paper, we continue to denote the superposition operator $\nabla\mathcal{R}$ by $\nabla R$ when there is no ambiguity; cf.~\Cref{sub:Notation}.

\begin{lemma}[Parametrization of $\mathring{\mathcal{K}}$]\label{lem:parametrization}
  Let $P\in\mathbb{R}^n$ be a closed polytope with non-empty interior, and $R$ be the Legendre function defined as in \Cref{thm:polylegendre}.
  We define the superposition operator $\nabla \mathcal{R}^*:L^\infty(\Omega;\mathbb{R}^n)\rightarrow \mathcal{K}$ by
  \begin{equation*}
    \nabla \mathcal{R}^*(\bm{\psi})(x):=\nabla R^*(\bm{\psi}(x)) \quad \text{f.a.e. } x\in \Omega.
  \end{equation*}
  Then $\mathring{\mathcal{K}}$ is the image of $\nabla \mathcal{R}^*$, in the sense that
  \begin{equation*}
    \mathring{\mathcal{K}} = \{\nabla \mathcal{R}^*(\bm{\psi})\mid\bm{\psi}\in L^\infty(\Omega;\mathbb{R}^n)\}.
  \end{equation*}
  This defines a parameterization of $\mathring{\mathcal{K}}$ by $L^\infty(\Omega;\mathbb{R}^n)$ because the underlying map is invertible.
\end{lemma}
\begin{proof}
  Let $\bm{p}\in \operatorname{int}(P)$ be arbitrary. Then, by \Cref{thm:polylegendre}, the gradient map $\nabla R^*$ is bijective with the inverse mapping $\nabla R$.
  Now, consider an arbitrary $\bm{\eta} \in \mathring{\mathcal{K}}$.
  By definition, we have
  \begin{equation*}
    0<\epsilon:=\operatorname*{ess\,inf}_{x\in\Omega}\big(\operatorname{dist}(\bm{\eta}(x),\partial P)\big).
  \end{equation*}
  This and the boundedness of $P$ imply that there exists a compact subset $C\subset \operatorname{int}(P)$ such that $\bm{\eta}(x) \in C$ for almost every $x \in \Omega$.
  By \Cref{thm:polylegendre}, $\nabla R$ is a homeomorphism from $\operatorname{int}(P)$ to $\mathbb{R}^n$.
  Therefore, the image of $C$ under $\nabla R$, denoted by $D:=\nabla R(C)$, is a compact subset of $\mathbb{R}^n$.

  We define $\bm{\psi}(x) := \nabla R(\bm{\eta}(x))\in D$ for almost every $x$ in $\Omega$.
  Since $\nabla R$ is continuous and $\bm{\eta}$ is measurable, $\bm{\psi}$ is a measurable function.
  Furthermore, the compactness of $D$ implies that $\bm{\psi}$ is essentially bounded; i.e., $\bm{\psi}\in L^\infty(\Omega;\mathbb{R}^n)$.
  Finally, we have from \Cref{thm:polylegendre} that $$\nabla \mathcal{R}^*(\bm{\psi})(x)=\nabla R^*\Big((\nabla R\big(\bm{\eta}\big(x)\big)\Big)=\bm{\eta}(x).$$
  The bijectivity of the mapping $\nabla R^*$ implies that the characterization of $\bm{\eta}$ by $\bm{\psi}$ is unique.
\end{proof}
\begin{remark}[Polytopes with empty interior]
  With a minor modification, the construction above can be applied to any lower-dimensional convex polytope $P$, including $\Delta^{n-1}$. In this case, the function $R$ is of Legendre type in the restricted topology of the affine hull of $P$. Similar to \Cref{lem:gibbs}, the gradient mapping $\nabla R^*$ will have a non-trivial invariant direction corresponding to the orthogonal complement of the affine hull. Therefore, the characterization of $\mathring{\mathcal{K}}$ by $\bm{\psi}\in L^\infty(\Omega;\mathbb{R}^n)$ is unique up to translation in the invariant direction.
  Equipped with the restricted topology, lower-dimensional polytopes can also be used as constraint sets, and the remaining analysis in this section still holds where $\nabla R$ is understood as an equivalence class in $\mathbb{R}^n/\operatorname{aff}(P)^\perp$. For clarity and simplicity, we focus only on convex polytopes with non-empty interiors.
\end{remark}
\begin{lemma}\label{lem:lsc-bregman}
  Let $P$ be a convex polytope with non-empty interior and $R$ be the Legendre function defined as in \Cref{thm:polylegendre}.
  Then the integral functional induced by the Legendre function,
  \begin{equation*}
    I_R:\bm{\eta} \in \mathcal{K}\mapsto \int_\Omega R(\bm{\eta})\dd x,
  \end{equation*}
  is strictly convex in $\mathring{\mathcal{K}}$ and lower-semicontinuous in $L^1$.
\end{lemma}
\begin{proof}
  The strict convexity of $I_R$ in $\mathring{\mathcal{K}}$ follows from the essential strict convexity of $R$ and the positivity of measure of $\{x\in\Omega:\bm{\eta}(x)\neq \bm{q}(x)\}$ for any two distinct $\bm{\eta},\bm{q}\in \mathring{\mathcal{K}}$ in $L^1$.
  Therefore, it is enough to verify the lower-semicontinuity of $I_R$.

  Let $\bm{\eta}\in\mathcal{K}$ be arbitrary and $\{\bm{\eta}_k\}\subset L^1(\Omega;\mathbb{R}^n)$ with $\bm{\eta}_k\to\bm{\eta}$ in $L^1$.
  We denote $L:=\liminf_k I_R(\bm{\eta}_k)$, which is finite since $R$ is bounded below and $\Omega$ has finite measure.
  We may first extract $\{\bm{\eta}_{k_\ell}\}$ with $I_R(\bm{\eta}_{k_\ell})\to L$, and then a further subsequence (still denoted $\{\bm{\eta}_{k_\ell}\}$) with $\bm{\eta}_{k_\ell}(x)\to\bm{\eta}(x)$ for almost every $x\in \Omega$.
  Since $R$ is lower-semicontinuous in $\mathbb{R}^n$ and bounded below, Fatou's lemma gives
  \begin{align*}
    I_R(\bm{\eta}) \leq \int_\Omega \liminf_\ell R(\bm{\eta}_{k_\ell})\,\dd x
    \leq \liminf_\ell I_R(\bm{\eta}_{k_\ell}) = L = \liminf_k I_R(\bm{\eta}_k).
  \end{align*}
\end{proof}
We are now ready to show that our iterates evolve in $\mathring{\mathcal{K}}\cap \mathcal{M}$.
The following theorem implies that the unconstrained mirror descent update $\bm{\eta}^{k+1/2}=\nabla R^*(\bm{\psi}^{k+1/2})$ belongs to $\mathring{\mathcal{K}}$.

\begin{theorem}[Interior Evolution]\label{thm:md-interior}
  Suppose that $F$ satisfies Assumption~\ref{asum:grad}.
  Let $\alpha>0$ and $R$ be a Legendre function with $P=\operatorname{dom}(R)$ defined as in \Cref{thm:polylegendre}.
  In addition, let $\bm{v}\in \mathring{\mathcal{K}}$ with $\nabla R(\bm{v})\in L^\infty(\Omega;\mathbb{R}^n)$.
  Then
  \begin{equation*}
    \min_{\bm{\eta}\in L^1(\Omega;\mathbb{R}^n)}\Big\{J(\bm{\eta}):=\int_\Omega \alpha \nabla F(\bm{v})\cdot \bm{\eta} + \mathcal{D}_R(\bm{\eta},\bm{v})\dd x\Big\}
  \end{equation*}
  is attained at a unique point $\bm{\eta}^\star\in \mathcal{K}$. Furthermore, $\bm{\eta}^\star$ is uniformly bounded away from $\partial P$, implying that
  $$\bm{\eta}^\star\in \mathring{\mathcal{K}}\text{ and }\nabla R(\bm{\eta}^\star)\in L^\infty(\Omega;\mathbb{R}^n).$$
\end{theorem}
\begin{proof}
  The proof consists of two parts: (1) existence of $\bm{\eta} ^\star$ in $ \mathcal{K}$ and uniqueness in $\mathring{\mathcal{K}}$; (2) $\nabla R(\bm{\eta}^\star)\in L^\infty(\Omega;\mathbb{R}^n)$.

  \textbf{Part 1}. The first part of the proof is analogous to \cite[Lemma 3.3]{simpl-math}.
  $\mathcal{K}$ is uniformly integrable and bounded in $L^1$ as $P$ is closed, bounded convex set in $\mathbb{R}^n$ and $\Omega$ has finite measure.
  The Dunford--Pettis theorem \cite[Theorem 2.4.5]{attouch2014variational} implies that $\mathcal{K}$ is weakly sequentially compact in $L^1(\Omega;\mathbb{R}^n)$.
  Recall that $F$ satisfies \Cref{asum:grad} and $\mathcal{D}_R$ is bounded below. This implies that $J$ admits a finite infimum.
  Let $\{\bm{\eta}_i\}\subset \mathcal{K}$ be an infimizing sequence of $J$.
  The weak sequential compactness implies that there exists a subsequence $\{\bm{\eta}_{i_k}\}$ that is $L^1$-weakly convergent to some point $\bm{\eta}^\star\in \mathcal{K}$.
  \Cref{lem:lsc-bregman} implies that $J$ is lower-semicontinuous in $L^1$, and hence
  \begin{equation*}
    J(\bm{\eta}^\star)\leq\lim\inf_{k\to \infty}J(\bm{\eta}_{i_k}).
  \end{equation*}
  Therefore, $\bm{\eta}^\star\in \mathcal{K}$ is a minimizer.
  Let $\mathcal{K}^\star$ be the set of minimizers, which is not empty.
  The strict convexity of $I_R$ in $\mathring{\mathcal{K}}$ implies that $J$ is also strictly convex, and hence $\mathcal{K}^\star\cap \mathring{\mathcal{K}}$ has at most one element.
  Therefore, the uniqueness will follow when $\mathcal{K}^\star\backslash \mathring{\mathcal{K}}$ is empty.

  \textbf{Part 2}. Since $L^1(\Omega;\mathbb{R}^n)$ is decomposable, we can interchange the minimization and integral \cite[Theorem 14.60]{Rockafellar1998}. Furthermore, the minimizer can be identified pointwise almost everywhere by the first-order optimality condition
  \begin{equation*}
    \bm{0}\in\alpha\nabla F(\bm{v}(x))+\partial R(\bm{\eta}(x))-\nabla R(\bm{v}(x)) ~\text{ f.a.e. }x\in\Omega.
  \end{equation*}
  Since $R$ is essentially smooth, essentially strictly convex, and coercive, this pointwise minimization problem has a unique solution. Also, the essential smoothness implies that the subdifferential of $R$ is an empty set at $\partial P$. Therefore, the solution $\bm{\eta}^\star(x)$ should lie in the interior of $\operatorname{dom}(R)=P$ for almost every $x$ in $\Omega$.
  Since $\nabla R$ is differentiable on $\operatorname{int}(P)$, we can rearrange terms to arrive at the identity
  \begin{equation*}
    \nabla R(\bm{\eta}(x))=\nabla R(\bm{v}(x)) - \alpha\nabla F(\bm{v}(x)).
  \end{equation*}
  Since $\nabla R(\bm{v}), \nabla F(\bm{v})\in L^\infty(\Omega;\mathbb{R}^n)$, the right hand side is uniformly bounded by $$M:=\alpha\|\nabla F(\bm{v})\|_{L^\infty(\Omega)} +\| \nabla R(\bm{v})\|_{L^\infty(\Omega)}$$
  This shows that $\nabla R(\bm{\eta}^\star)\in L^\infty(\Omega;\mathbb{R}^n)$.
  Using the inverse relation \Cref{eq:polymap}, we have $\bm{\eta}^\star=\nabla R^*(\bm{\psi})$ for some $\bm{\psi}\in L^\infty(\Omega;\mathbb{R}^n)$ with $\|\bm{\psi}\|_{L^\infty(\Omega)}\leq M$.
  Then \Cref{lem:parametrization} implies that $\bm{\eta}^\star\in \mathring{\mathcal{K}}$.
\end{proof}
Finally, we show that the mirror descent method produces strictly feasible iterates $\{\bm{\eta}^k\}\subset \mathring{\mathcal{K}}\cap\mathcal{M}$.
\begin{corollary}[Projected Mirror Descent]\label{cor:bregman-projection}
  Let $P$ be a convex polytope with non-empty interior and $R$ be the Legendre function defined as in \Cref{thm:polylegendre}.
  Suppose that there exists $\bm{\eta}$ such that
  \begin{equation*}
    \bm{\eta}(x)\in P\text{ a.e. }x\in \Omega,\; \int_\Omega W\bm{\eta}\dd x\prec \bm{b}.
  \end{equation*}
  If $F$ satisfies \Cref{asum:grad}, $\bm{v}\in \mathring{\mathcal{K}}$ with $\nabla R(\bm{v})\in  L^\infty(\Omega;\mathbb{R}^n)$, and $\alpha>0$,
  then
  \begin{equation*}
    \begin{aligned}
      \min_{\bm{\eta}\in L^1(\Omega;\mathbb{R}^n)} & \int_\Omega\alpha\nabla F(\bm{v})\cdot\bm{\eta}+\mathcal{D}_R(\bm{\eta},\bm{v})\dd x \\
      \subto                                       & \int_\Omega W\bm{\eta}\preccurlyeq \bm{b}
    \end{aligned}
  \end{equation*}
  has a unique minimizer $\bm{\eta}^\star$. Furthermore, $\bm{\eta}^\star$ is uniformly bounded away from $\partial P$ so that
  $$\bm{\eta}^\star\in \mathring{\mathcal{K}}\text{ and }\nabla R(\bm{\eta}^\star)\in L^\infty(\Omega;\mathbb{R}^n).$$
  Consequently, the minimizer $\bm{\eta}^\star$ obtained in $L^1$ is also the unique minimizer of the problem when embedded in the $L^\infty$-topology.
\end{corollary}
\begin{proof}
  The existence and uniqueness of the minimizer $\bm{\eta}^\star \in \mathcal{K}\cap \mathcal{M}$ can be shown similarly to \Cref{thm:md-interior} by observing that $\mathcal{K}\cap \mathcal{M}$ is nonempty, closed, convex, and uniformly integrable.

  The Bregman divergence is strictly convex and lower-semicontinuous with respect to the first argument, and the global constraint is linear.
  Also, the existence of a feasible point $\bm{\eta}$ satisfying the strict inequality corresponds to the generalized Slater condition used in infinite-dimensional optimization; see, e.g., \cite[Chap. 3]{Bonnans2000}.
  Then the Fenchel--Rockafellar duality \cite{rockafellar1974conjugate} implies that the strong duality holds.
  Now, consider the following Lagrangian over $L^1(\Omega;\mathbb{R}^n)\times \mathbb{R}^m$:
  \begin{equation*}
    \mathcal{L}(\bm{\eta},\bm{\mu})=\int_\Omega \alpha\nabla F(\bm{v})\cdot\bm{\eta}+\mathcal{D}_R(\bm{\eta},\bm{v})\dd x + \bm{\mu} \cdot \Big(\int_\Omega W\bm{\eta}\dd x - \bm{b}\Big).
  \end{equation*}
  The non-negativity of the Bregman divergence implies that the Lagrangian has a saddle point $(\bm{\eta}^\star,\bm{\mu}^\star)\in (\mathcal{K}\cap\mathcal{M})\times \mathbb{R}^m$ with finite $p^\star=\mathcal{L}(\bm{\eta}^\star,\bm{\mu}^\star)$. The strong duality implies that
  \begin{align*}
    \max_{\bm{\mu}\in \mathbb{R}^m_{\geq 0}}\min_{\bm{\eta}\in L^1(\Omega;\mathbb{R}^n)}\mathcal{L}(\bm{\eta},\bm{\mu})
    =\mathcal{L}(\bm{\eta}^\star,\bm{\mu}^\star).
  \end{align*}

  Similarly to the previous, we can solve the inner-minimization problem point-wise, so that the minimizer $\bm{\eta}_{\bm{\mu}}$ for a given $\bm{\mu} \in \mathbb{R}^m_{\geq 0}$ is given by
  \begin{equation}\label{eq:optim-lagrangian}
    \nabla R(\bm{\eta}_{\bm{\mu}})(x) = \nabla R(\bm{v})(x)-\alpha\nabla F(\bm{v})(x)-W^\top \bm{\mu} ~\text{ f.a.e. }x\in \Omega.
  \end{equation}
  Since $\bm{\eta}_{\bm{\mu}^\star}=\bm{\eta}^\star$, we have $\nabla R(\bm{\eta}^\star)\in L^\infty(\Omega;\mathbb{R}^n)$ by taking $\bm{\eta}_{\bm{\mu}}=\bm{\eta}^\star$ and $\bm{\mu}=\bm{\mu}^\star$.
  Finally, the boundedness of $P$ implies that $\bm{\eta}^\star\in L^\infty(\Omega;\mathbb{R}^n)$. Thus, $\bm{\eta}^\star$ is also the unique minimizer of the problem when embedded in the $L^\infty$-topology.
\end{proof}

\begin{remark}[Pure Phases]
  Note that the range of the gradient mapping is open: $\nabla R^*(\mathbb{R}^n) = \operatorname{int}(P)$.
  Theoretically, this implies that a pure phase $\bm{v}_i$ is only attainable in the limit as $|\bm{\psi}| \to \infty$.
  However, the specific structure of $\nabla R^*$ (involving exponential functions like softmax) ensures that $\operatorname{dist}(\nabla R^*(\bm{\psi}), \partial P)$ decays exponentially with respect to the magnitude of $\bm{\psi}$.
  Numerically, this saturation is rapid; a latent variable $\bm{\psi}$ with magnitude of order $\mathcal{O}(10)$ is typically sufficient to approximate a pure-phase solution to machine precision.
\end{remark}
\begin{remark}[Global Convex Constraint]
  Although we have only considered the linear global constraint $\int_\Omega W\bm{\eta}\dd x\preccurlyeq \bm{b}$, this framework can be generalized to accommodate convex global constraints of the form $\int_\Omega g(\bm{\eta})\dd x\preccurlyeq \bm{b}$, where $g:\mathbb{R}^n\to\mathbb{R}^m$ is a lower semicontinuous convex function.
\end{remark}

\section{Practical Considerations}
In this section, we address key practical considerations for the numerical implementation of the proposed optimization framework. First, we detail the mathematical treatment for the rotation of anisotropic materials, incorporating the polytopal material design domain $P$. Subsequently, we outline the computational procedures for efficiently evaluating the Bregman projection.
\subsection{Rotation of Anisotropic Materials}\label{subsec:aniso}
We consider a composite design domain consisting of two phases: an isotropic phase and an anisotropic phase.
Let $\mathsf{C}_{iso}$ and $\mathsf{C}_{ani}$ denote their respective material tensors.
The anisotropic tensor is permitted to rotate at ${\texttt{nv}}-1$ discrete angles $\{\theta_i\}_{i=1}^{{\texttt{nv}}-1}$, uniformly distributed in $[0,2\pi)$:
$$\theta_i = \frac{2\pi(i-1)}{{\texttt{nv}}-1}, \quad i=1,\ldots,{\texttt{nv}}-1.$$
To model this, we construct a polytope $P$ defined by a regular $({\texttt{nv}}-1)$-gon base centered at the origin and an apex at $(0,0,1)$.
The $i$-th base vertex is placed at
$$\bm{v}_i = \big(\cos\theta_i,\,\sin\theta_i,\,0\big),$$
so that the vertices are uniformly distributed on the unit circle in the base plane.
The vertices of the base represent the rotated anisotropic phases, while the apex represents the isotropic phase.
Then the design variable $\bm{\eta}(x)=(a(x),b(x),s(x))\in P$ at a point $x\in \Omega$ encodes the material phase and orientation at that point.
The effective material tensor is defined by the SIMP-like interpolation:
$$\mathsf{C}(\bm{\eta})=\mathsf{C}(a,b,s) := s^p \mathsf{C}_{iso} + r^p \mathsf{C}_{ani}(\phi) = s^p\mathsf{C}_{iso}+r^p\big[ \mathbf{R}(\phi)\, \mathsf{C}_{ani}\, \mathbf{R}(-\phi) \big],$$
where $(r,\phi)$ are the polar coordinates of the planar design variable $(a,b)$, and we omit the explicit dependence on $x$ for brevity.
This is equivalent to: (1) rotate the physical quantity of interest (strain, current, gradient) to the reference material coordinate, (2) apply the material tensor, (3) rotate back to the physical coordinate.
The penalization power $p$ should be chosen properly to ensure differentiability along the $s$-axis (isotropic phase).
For instance, $p=2$ is sufficient for the second-order tensor interpolation in diffusion, and $p=4$ is sufficient for the fourth-order tensor interpolation in elasticity.
Within the $({\texttt{nv}}-1)$-gonal base, $r=1$ holds only at the vertices (pure anisotropic phases).
Consequently, the combination of this interpolation and a global mass constraint effectively penalizes intermediate design variables $(a,b,s)$ that do not coincide with the vertices of $P$.
In three-dimensional cases, the base can be parametrized by (scaled-)quaternions to represent the full (scaled-)rotation group $SO(3)$ with three degrees of freedom, and the interpolation can be defined similarly by replacing $\mathbf{R}(\phi)$ with the corresponding rotation matrix.

\paragraph{Exploiting periodicity}
For many physical tensors, such as elasticity and diffusion tensors, the material property is invariant under a $\pi$-rotation:
$$\mathsf{C}_{ani}(\theta) = \mathsf{C}_{ani}(\theta+\pi).$$
In this case, distinct orientations need only be sampled over the half-period $[0,\pi)$, so the ${\texttt{nv}}-1$ discrete angles are taken as
$$\theta_i = \frac{\pi(i-1)}{{\texttt{nv}}-1}, \quad i=1,\ldots,{\texttt{nv}}-1.$$
To encode these angles without breaking the symmetry of the regular $({\texttt{nv}}-1)$-gon, we use a \emph{double-angle} encoding and place the $i$-th base vertex at
$$\bm{v}_i = \big(\cos(2\theta_i),\,\sin(2\theta_i),\,0\bigr).$$
Since $2\theta_i$ ranges over $[0,2\pi)$, the vertices still form a regular $({\texttt{nv}}-1)$-gon — the polytope geometry is unchanged.
The material orientation is then recovered from the polar angle $\phi$ via $\theta = \phi/2$, and the interpolation becomes
\begin{equation}\label{eq:C-periodic}
  \mathsf{C}(\bm{\eta})=\mathsf{C}(a,b,s):= s^p\mathsf{C}_{iso}+r^p\big[ \mathbf{R}(\phi/2)\, \mathsf{C}_{ani}\, \mathbf{R}(-\phi/2) \big].
\end{equation}
This encoding intrinsically avoids redundancy: as $\phi$ traverses $[0,2\pi)$, the material orientation $\theta = \phi/2$ traverses $[0,\pi)$ exactly once, yielding a bijection between base vertices and distinct material phases, while preserving the proximity between physical angles.

\subsection{Bregman Projection}\label{subsec:projection}
The projection step, $\bm{\eta}^{k+1}=\textrm{Proj}_R(\bm{\eta}^{k+1/2})$ (cf. \eqref{eq:update_with_proj}), is equivalent to solving the finite-dimensional dual problem:
\begin{equation}\label{eq:dual-problem}
  \max_{\bm{\mu}\succcurlyeq\bm{0}}\Big\{g(\bm{\mu})=-\int_\Omega R^*(\bm{\psi}^{k+1/2}-W^\top \bm{\mu})\dd x-\bm{\mu}^\top \bm{b}\Big\}.
\end{equation}
For the specific case of a single global constraint (e.g., total mass), this reduces to a one-dimensional root-finding problem, for which we employ the Illinois algorithm as described in \cite{simpl-engrg}.

As the intermediate design variable $\bm{\eta}^{k+1/2}$ converges toward pure phases, the optimization landscape becomes increasingly ill-conditioned.
While a standard projected Newton's method suffices when the transition region (or ``gray region'') is sufficiently wide, it becomes unstable as the solution sharpens and the Hessian degenerates.
Although the Bregman--Dykstra projection on convex sets \cite{Bauschke2000} is theoretically guaranteed to converge, we observe that its convergence rate significantly degrades in this regime.
Developing a theoretically robust and efficient numerical method for this regime remains a subject for future work.

\section{Numerical Experiments}
In this section, we apply \Cref{alg:pmd-gbb} to multi-material topology optimization problems.
In the following, we use the adaptive step size by the generalized Barzilai–Borwein method (GBB) following \cite{simpl-engrg}:
\begin{equation}\label{eq:gbb}
  \alpha_k := \left|\frac{\int_\Omega (\bm{\eta}^k-\bm{\eta}^{k-1})\cdot(\bm{\psi}^k-\bm{\psi}^{k-1})\dd x}{\int_\Omega (\bm{\eta}^k-\bm{\eta}^{k-1})\cdot\Big(\nabla F(\bm{\eta}^k)-\nabla F(\bm{\eta}^{k-1})\Big)\dd x}\right|.
\end{equation}
Additionally, we use the generalized Armijo condition
\begin{equation}\label{eq:armijo}
  F(\bm{\eta}^{k+1})\leq F(\bm{\eta}^k)+c_1\int_\Omega \nabla F(\bm{\eta}^k)\cdot(\bm{\eta}^{k+1}-\bm{\eta}^k)\dd x.
\end{equation}
Here, we choose $c_1=10^{-4}$. If we assume \Cref{asum:grad} with $L$-Lipschitz continuity of $\nabla F$, then the backtracking line search terminates in a finite number of iterations with monotonically decreasing objective values.
For the Bregman projection step, we used the Dykstra projection algorithm \cite{Bauschke2000}.
The numerical experiments were implemented using the finite element libraries MFEM \cite{andrej2024mfem,tzanio2010mfem} and NGSolve \cite{Schoberl2014ngsolve}.

\begin{remark}[Stopping Criterion]
  By construction, the iterates satisfy primal feasibility $\bm{\eta}^k\in\mathcal{A}=\mathcal{K}\cap\mathcal{M}$ for all $k$. Furthermore, the Bregman projection step implicitly generates a dual-feasible Lagrange multiplier $\bm{\mu}^k$ (cf. \eqref{eq:optim-lagrangian}). Consequently, convergence verification reduces to checking stationarity and complementary slackness.
  To this end, we define the scaled successive latent update $\bm{d}^{k}$ as a proxy for the total gradient:
  \begin{equation*}
    \bm{d}^{k}:=\frac{\bm{\psi}^k-\bm{\psi}^{k+1}}{\alpha_k}=\nabla F(\bm{\eta}^k)+\frac{1}{\alpha_k}W^\top \bm{\mu}^k.
  \end{equation*}
  Recall that for a stationary point $\bm{\eta}^\star\in\mathcal{A}$ with multiplier $\bm{\mu}^\star\in\mathbb{R}^r_{\geq 0}$, the following variational inequality holds:
  \begin{equation}
    \int_\Omega (\nabla F(\bm{\eta}^\star)+W^\top \bm{\mu}^\star)\cdot(\bm{v}-\bm{\eta}^\star)\dd x \geq 0 \quad \forall \bm{v}\in \mathcal{K}.
  \end{equation}
  Exploiting the pointwise convexity of $\mathcal{K}$ and the linearity of the inner product, the minimum of the variational term over $\mathcal{K}$ is attained at the vertices. We can thus quantify the violation of stationarity at iterate $k$ by the gap function:
  \begin{equation*}
    g(\bm{\eta}^k, \bm{d}^{k}) := \int_\Omega \min_{\bm{v}\in \operatorname{vert}(P)} \left\{ \bm{d}^{k} \cdot (\bm{v}-\bm{\eta}^k) \right\} \dd x.
  \end{equation*}
  Note that this integral is always non-positive. We explicitly define the residual for the stopping criterion as the magnitude of this gap:
  \begin{equation}\label{eq:KKT_estimator}
    \texttt{res}_{k} := \int_\Omega \left| \min_{\bm{v}\in \operatorname{vert}(P)} \left\{ \bm{d}^{k} \cdot (\bm{v}-\bm{\eta}^k) \right\} \right| \dd x.
  \end{equation}
  The algorithm terminates when $\texttt{res}_{k}$ or $\texttt{res}_{k}/\texttt{res}_{0}$ fall below a specified tolerance.
\end{remark}

\subsection{Compliance Minimization with Isotropic Materials}\label{subsec:num-iso}
\begin{figure}
  \centering
  \begin{tabular}{cc}
    \includegraphics[width=0.5\textwidth, trim=0 -150pt 0 0]{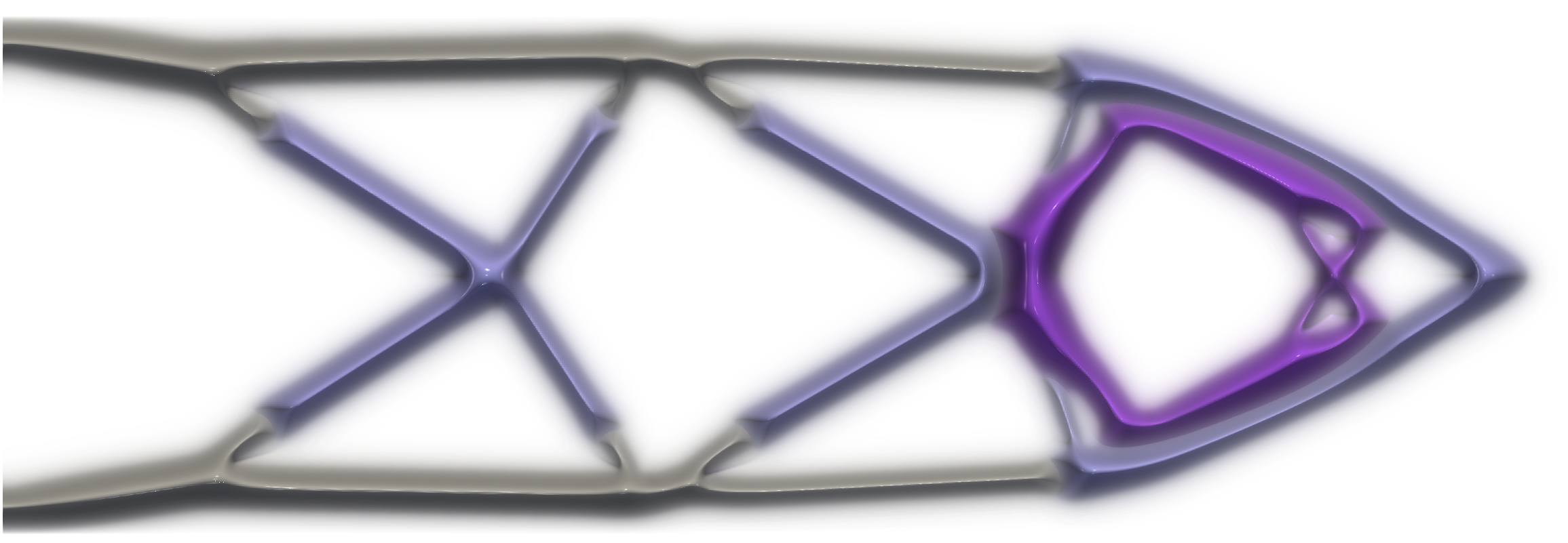} &
    \vspace{0.3em}\includegraphics[width=0.35\textwidth]{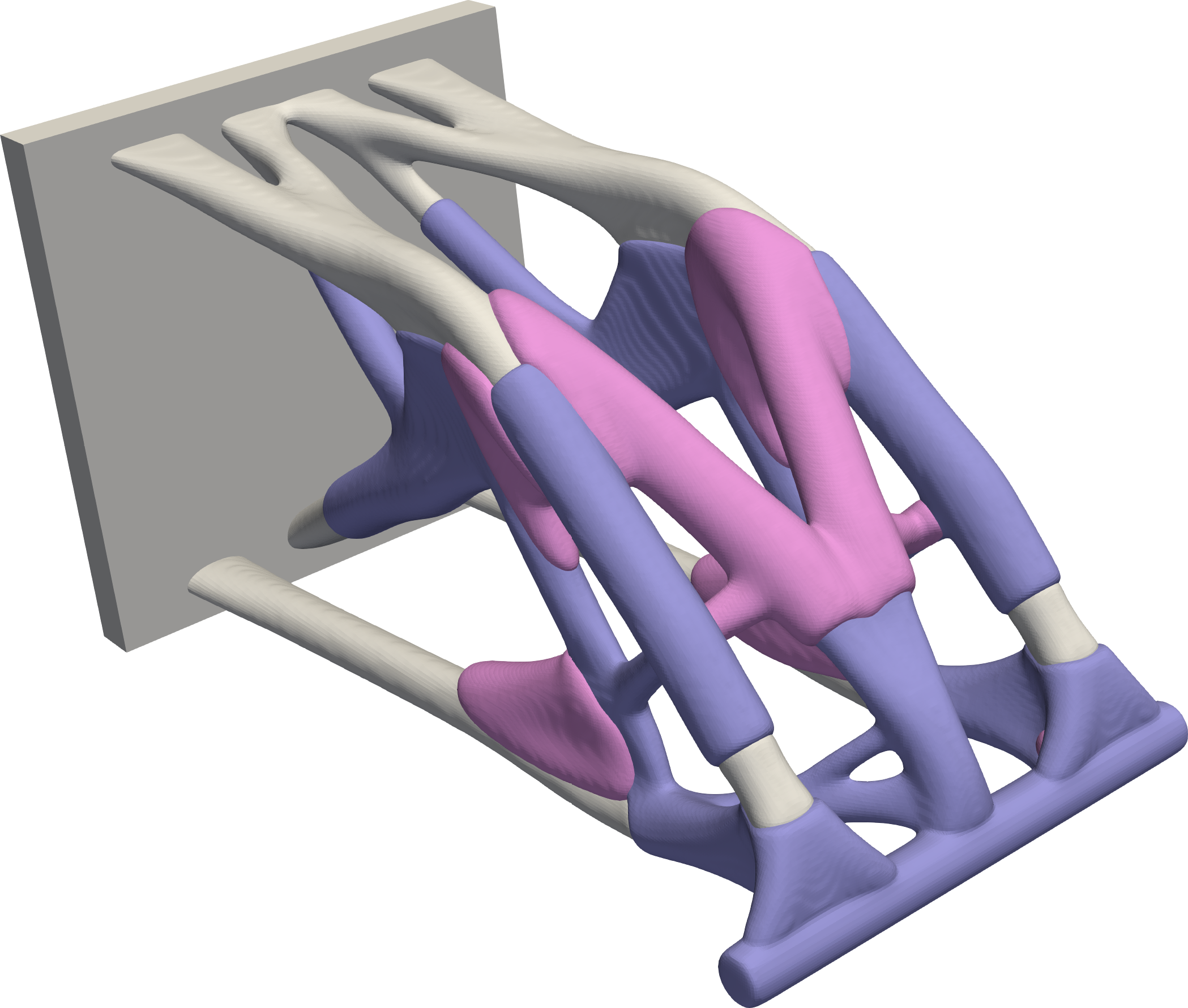}          \\
    (a)                                                                                  & (b)
  \end{tabular}
  \caption{Optimized designs for compliance minimization using isotropic materials with Young's moduli $E_1=10^{-6}$, $E_2=1$ (pink), $E_3=3$ (purple), $E_4=5$ (grey), and Poisson's ratio $\nu=0.3$. (a) 2D example after 31 iterations, (b) 3D example after 47 iterations.\label{fig:compliance-min-design}}
\end{figure}
In this example, we consider the compliance minimization problem over the domain $\Omega=(0,3)\times(0,1)$:
\begin{subequations}\label{eq:comp-min}
  \begin{align}
    \min_{\bm{\eta}\in L^\infty(\Omega;\mathbb{R}^n)} &
    \int_\Omega \bm{f}\cdot\bm{u}(\bm{\eta})\dd x                                                                                                                                                                                                                      \\
    \subto
                                                      &
    \int_\Omega \epsilon^2\nabla \tilde{\bm{\eta}}:\nabla \tilde{\bm{q}}+\tilde{\bm{\eta}}\cdot\tilde{\bm{q}}\dd x=\int_\Omega \bm{\eta}\cdot\tilde{\bm{q}}\dd x\quad\forall \bm{q}\in H^1_F(\Omega), \label{eq:pde_filter}                                            \\
                                                      & \int_\Omega \mathsf{C}\big(\tilde{\bm{\eta}}\big):\varepsilon\big(\bm{u}(\tilde{\bm{\eta}})\big):\varepsilon(\bm{v})\dd x=\int_\Omega \bm{f}\cdot\bm{v}\dd x\quad\forall \bm{v}\in H^1_D(\Omega;\mathbb{R}^2), \\
                                                      & \bm{\eta}(x)\in \Delta^{n-1}\text{ a.e. }x\in\Omega,                                                                                                                                                           \\
                                                      & \int_\Omega \eta_i\dd x\leq 0.18\quad\text{for }i=2                                                                                                                                                            \\
                                                      & \int_\Omega \eta_i\dd x\leq 0.36\quad\text{for }i=3,4.
  \end{align}
\end{subequations}
Here,\begin{equation*}
  \begin{aligned}
    \bm{f}(x)                  & :=\begin{cases}
                                     (0,-1) & \text{ when }|x-(2.9,0.5)|\leq 0.05, \\
                                     \bm{0} & \text{ otherwise.}
                                   \end{cases},                                                                            \\
    H^1_F(\Omega;\mathbb{R}^n) & :=\{\tilde{\bm{\eta}}\in H^1(\Omega;\mathbb{R}^n):\tilde{\bm{\eta}}|_{\Gamma_F}=\bm{0}\},\quad\Gamma_F=(0,3)\times\{0,1\}, \\
    H^1_D(\Omega;\mathbb{R}^2) & :=\{\bm{u}\in H^1(\Omega;\mathbb{R}^2):\bm{u}|_{\Gamma_D}=\bm{0}\},\quad\Gamma_D=\{0\}\times(0,1),
  \end{aligned}
\end{equation*}
$\epsilon=0.06/2\sqrt{3}$, and $\mathsf{C}(\tilde{\eta})$ is the effective elasticity tensor defined by the Poisson ratio $\nu$ and the effective Young's modulus:
\begin{equation*}
  E_{\text{eff}}^{(j)}=E_{\text{eff}}^{(j-1)}\Big[1-\Big(\frac{\tilde{\eta}_j}{S_j}\Big)^p\Big]+E_j\Big(\frac{\tilde{\eta}_j}{S_j}\Big)^p\text{ a.e. }x\in\Omega,\;j=2,\ldots,n,
\end{equation*}
where $S_j=\sum_{i=1}^j\tilde{\eta}_i$ is the cumulative density, $p=3$ is the SIMP penalty parameter, and $E_j$ is Young's modulus of the $j$-th material phase.
We choose the Poisson ratio $\nu=0.3$ for all materials, $E_1=10^{-6}$ for the void material, and set $E_2=1$, $E_3=3$, and $E_4=5$ to represent the weak, normal, and strong materials, respectively.
The volume constraints are specifically partitioned so that the weakest material ($i=2$) is allocated exactly half the volume of the more robust materials ($i=3, 4$), ensuring their collective contribution satisfies the prescribed $30\%$ total volume threshold.
For the numerical discretization, we employ a uniform rectangular mesh $\mathcal{T}_h$ with a side length of $h=2^{-8}$. The indicator space is discretized using piecewise constant elements, while both the filtered indicator space and the displacement space are approximated using continuous piecewise bilinear elements.
The optimal design for each interpolation scheme is depicted in \Cref{fig:compliance-min-design}, and the histories of the objective function $F(\bm{\eta}^k)$, the gap residual $\texttt{res}_k$, and the step size $\alpha_k$ are shown in \Cref{fig:compliance-min-history}.
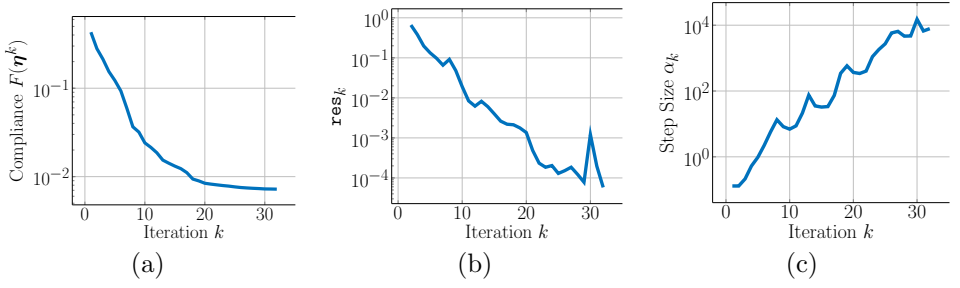
\begin{figure}
  \centering
  \begin{tabular}{ccc}
    \resizebox{0.3\textwidth}{!}{
    \begin{tikzpicture}[font=\Huge]
        \begin{axis}[
            axis y line* = left,
            width = 12cm,
            view = {45}{65},
            xlabel={Iteration $k$},
            ylabel={Compliance $F(\bm{\eta}^k)$},
            label style={font=\Huge},
            xmajorgrids, ymajorgrids,
            xtick = {0, 10, 20, 30, 40, 50},
            ymode = log,
            legend style={at={(0.98,0.98)}, anchor=north east, inner sep=6pt, font=\huge, fill=white, fill opacity=0.8, draw opacity=1, text opacity=1},
          ]

          \addplot[color1, line width=1.5mm] table [x expr=\coordindex+1, y index={1}, col sep=comma] {plots/result_intp1_r5.csv};
        \end{axis}
      \end{tikzpicture}} &
    \resizebox{0.3\textwidth}{!}{
    \begin{tikzpicture}[font=\Huge]
        \begin{axis}[
            axis y line* = left,
            width = 12cm,
            view = {45}{65},
            xlabel={Iteration $k$},
            ylabel={$\texttt{res}_k$},
            label style={font=\Huge},
            xmajorgrids, ymajorgrids,
            xtick = {0, 10, 20, 30, 40, 50},
            ymode = log,
            legend style={at={(0.98,0.98)}, anchor=north east, inner sep=6pt, font=\huge, fill=white, fill opacity=0.8, draw opacity=1, text opacity=1},
          ]

          \addplot[color1, line width=1.5mm] table [x expr=\coordindex+1, y index={3}, col sep=comma] {plots/result_intp1_r5.csv};
        \end{axis}
      \end{tikzpicture}} &
    \resizebox{0.3\textwidth}{!}{
    \begin{tikzpicture}[font=\Huge]
        \begin{axis}[
            axis y line* = left,
            width = 12cm,
            view = {45}{65},
            xlabel={Iteration $k$},
            ylabel={Step Size $\alpha_k$},
            label style={font=\Huge},
            xmajorgrids, ymajorgrids,
            xtick = {0, 10, 20, 30, 40, 50},
            ytick = {10^-2, 10^0, 10^2, 10^4, 10^6, 10^8},
            ymode = log,
            legend style={at={(0.98,0.98)}, anchor=north east, inner sep=6pt, font=\huge, fill=white, fill opacity=0.8, draw opacity=1, text opacity=1},
          ]

          \addplot[color1, line width=1.5mm] table [x expr=\coordindex+1, y index={2}, col sep=comma] {plots/result_intp1_r5.csv};
        \end{axis}
      \end{tikzpicture}}              \\
    (a)                             & (b) & (c)
  \end{tabular}
  \caption{Convergence history for (a) compliance $F(\bm{\eta}^k)$, (b) gap residual $\texttt{res}_k$, and (c) step size $\alpha_k$ for 2D compliance minimization with isotropic materials.\label{fig:compliance-min-history}}
\end{figure}

Combined with the adaptive step size \eqref{eq:gbb} and the generalized Armijo backtracking linesearch \eqref{eq:armijo}, the optimal 2D design is obtained after 31 iterations with the final relative gap residual $\texttt{res}_k/\texttt{res}_0\leq 10^{-4}$.
We note that the Armijo condition passed without failure for most of the optimization process with at most 3 backtracking steps.
This indicates that the GBB step size \eqref{eq:armijo} is well-suited for this problem.
The step size $\alpha_k$ grows almost exponentially, while achieving a monotonic decrease in the objective function.
We note that the Armijo condition does not guarantee a monotonic decrease of the gap residual $\texttt{res}_k$.

We also apply \Cref{alg:pmd-gbb} to a three-dimensional analogue of the problem above, posed on the cantilever domain $\Omega=(0,2)\times(0,1)\times(0,1)$ with the same materials, SIMP penalty $p=3$. The clamped face is $\Gamma_D=\{0\}\times(0,1)\times(0,1)$, and the filter Dirichlet boundary is taken to be the lateral surface $\Gamma_F=(0,2)\times\partial\big((0,1)\times(0,1)\big)$. The body force is concentrated in a cylindrical region of radius $0.05$ oriented along the $x_2$-axis near the tip:
\begin{equation*}
  \bm{f}(x):=\begin{cases}
    (0,0,-1) & \text{ when }(x_1-1.9)^2+(x_3-0.1)^2\leq 0.05^2, \\
    \bm{0}   & \text{ otherwise.}
  \end{cases}
\end{equation*}
In addition, we fixed the design variable $\eta_2$ at the loading region to ensure the presence of the intermediate material phase. For the volume constraints, we enforce $\int_\Omega \eta_{i+1}\dd x \leq b_i$ with $b_1=0.02|\Omega|$, $b_2=b_3=0.04|\Omega|$.
We discretize $\Omega$ with a uniform hexahedral mesh of side length $h=2^{-7}$, using piecewise constant elements for the indicator space and continuous piecewise trilinear elements for the filtered indicator and displacement spaces. Since the solution is expected to be symmetric with respect to the plane $x_2=0.5$, we only optimize over the half domain $\Omega \cap \{x_2\leq 0.5\}$, and used symmetry boundary conditions on the plane $x_2=0.5$ for both the filter and the elasticity equations.
The optimized design, obtained after 47 iterations with final gap residual $6.25057\times 10^{-6}$, is shown in \Cref{fig:compliance-min-design}(b).

\subsection{Compliance Minimization with Orthotropic Material}\label{subsec:num-aniso}
The next example demonstrates the usage of polytopal constraints for compliance minimization with orthotropic material. We consider a single orthotropic material phase with discrete orientations, in addition to the void phase.
The orthotropic material is defined by Young's moduli $E_x=5.0,\;E_y=0.5$ and Poisson's ratio $\nu_{xy}=0.3$ along the reference $x$- and $y$-axes, respectively.
The shear modulus is determined by $G_{xy} = \frac{\sqrt{E_x E_y}}{2(1+\sqrt{\nu_{xy}\nu_{yx}})}$, where $\nu_{yx} = \nu_{xy} E_y / E_x$ follows from the symmetry of the compliance matrix
With 16 discrete angles, $\theta_i=\frac{i\pi}{8}$ for $i=0,...,7$, we consider $P=\operatorname{conv}(\{(0,\cos(2\theta_i),\sin(2\theta_i))\}_{i=0}^{7}\cup \{(1,0,0)\})$.
The design variable $\bm{\eta}(x)=(\eta_1,\eta_2,\eta_3)^\top=(\eta_{void},r\cos(2\theta), r\sin(2\theta))^\top\in \mathbb{R}^3$ is defined such that $\eta_1$ represents the void phase, while $\eta_2$ and $\eta_3$ encode the orientation of the orthotropic material.
For the volume constraint, we enforce minimum void volume, $\int_\Omega \eta_1\dd x \geq 0.3|\Omega|$.
We used the same domain, boundary conditions, loading, and discretization as in the previous example.
With initial latent variable $\bm{\psi}=(0,0.1,0)^\top$, the optimal design is obtained after 38 iterations with $\texttt{res}_k/\texttt{res}_0\leq 10^{-12}$.
Here, we used a smaller stopping tolerance as the initial design is mostly void, leading to massive gradient magnitudes at the beginning of the optimization process. The final absolute gap residual is $\texttt{res}_k=4.8\times 10^{-6}$ with compliance $F(\bm{\eta}^k)=0.00277658$ after 35 iterations.
The final design is depicted in \Cref{fig:aniso}, where the colored lines indicate the material $x$-coordinates.
We observe alignment between the anisotropic material and the bulk structural features, consistent with the literature \cite{Stegmann2005,Hvejsel2011}.
\begin{figure}
  \centering
  \begin{tabular}{cc}
    \begin{overpic}[width=0.57\textwidth]{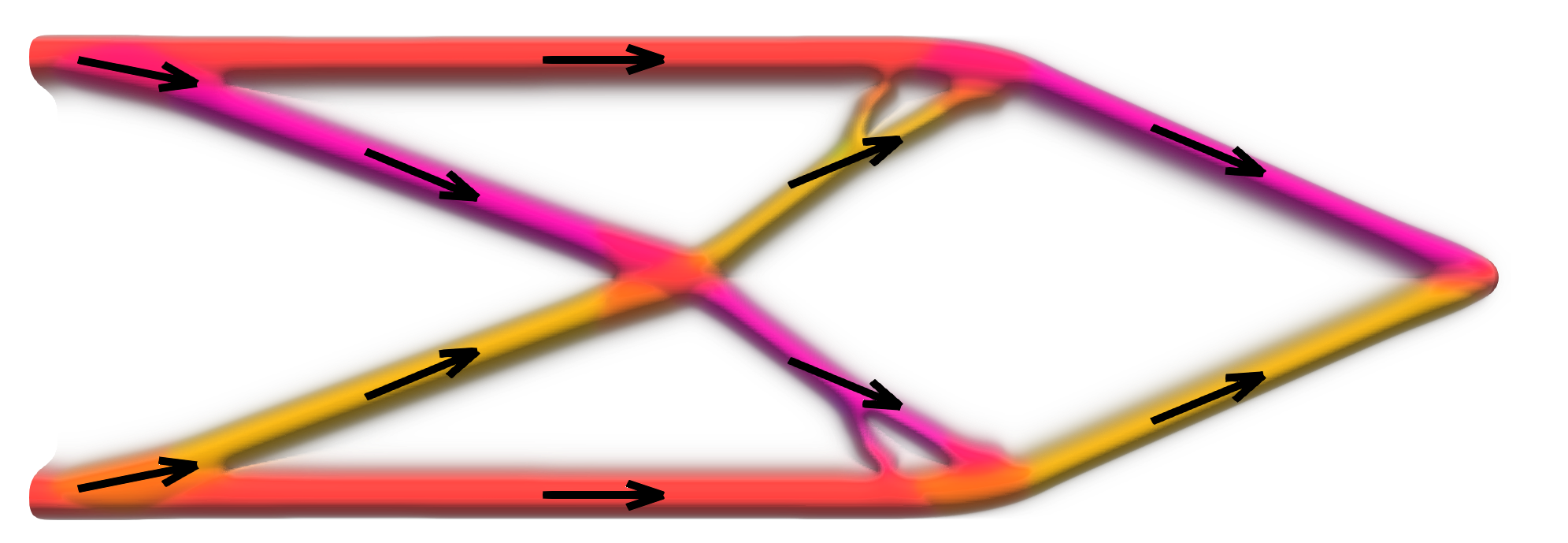}
    \end{overpic}
        &
    \hspace{3em}\includegraphics[width=0.22\textwidth]{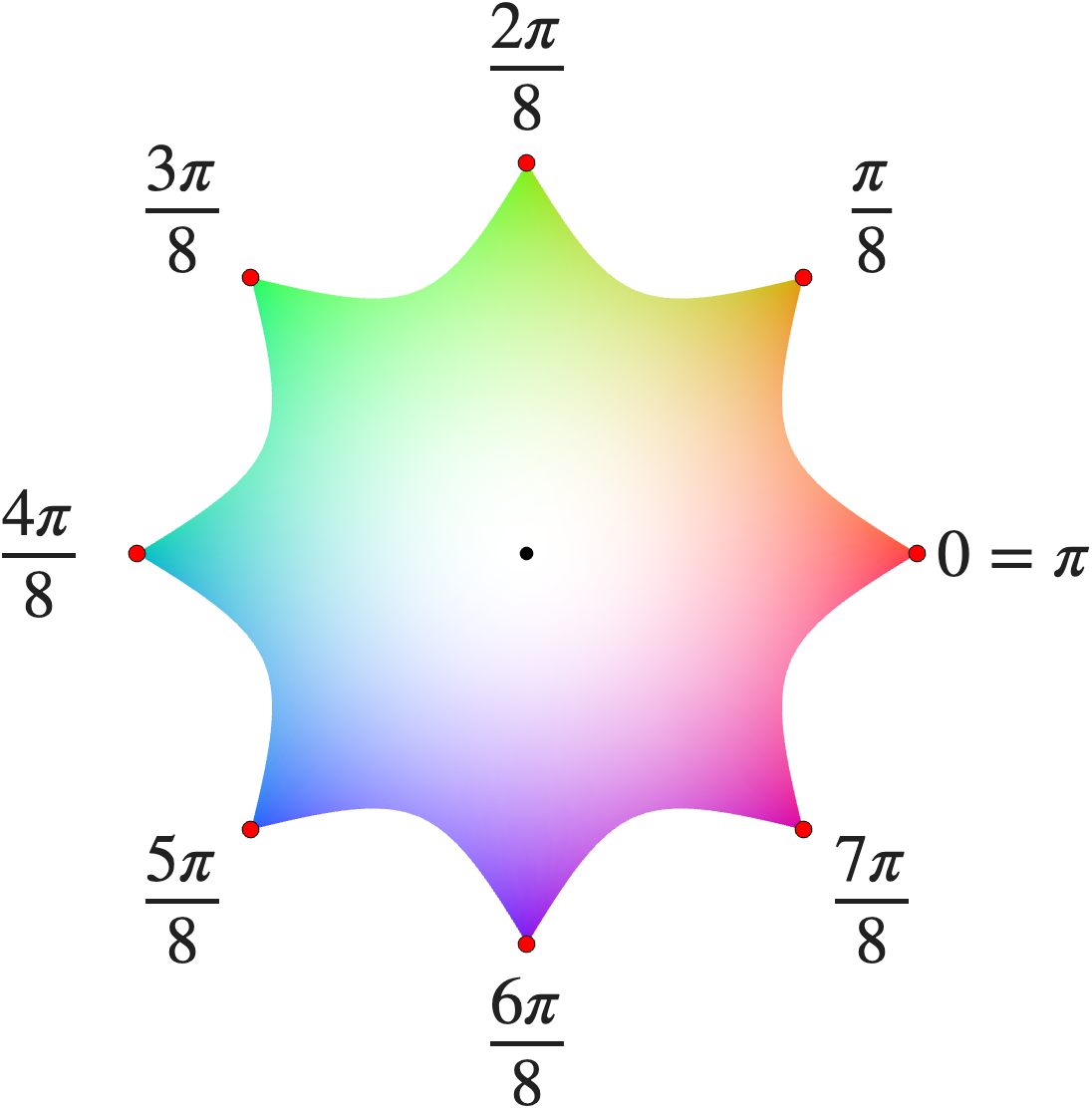} \\
    (a) & \hspace{3em}(b)
  \end{tabular}
  \caption{(a) Optimized design for compliance minimization using orthotropic material with 8 discrete angles, Young's moduli $E_x=5.0$, $E_y=0.5$, and Poisson's ratio $\nu_{xy}=0.3$. The optimal design is obtained after 35 iterations with $F(\bm{\eta})=0.00277658$. (b) Angular color map scaled by the effective Young's moduli with $p=4$, cf. \eqref{eq:C-periodic}.\label{fig:aniso}
  }
\end{figure}

\subsection{Maximizing magnetic flux in electric motor}
As a final example, we consider the optimization of a two-dimensional cross-section of an electric motor.
We aim at finding an optimal distribution of ferromagnetic material and permanent magnets in order to maximize the magnetic flux passing from the rotor (inner part) to the stator (outer part). Since permanent magnets have a high cost, their total volume typically needs to be constrained. The computational domain is depicted in Figure \ref{fig:motor}(a) with the design region $\Omega_d$ highlighted. For more details on this particular application, we refer the reader to \cite{CherriereGanglKrenn2025}. As in \cite{CherriereGanglKrenn2025}, the effect of the stator geometry is neglected, and it is modeled as pure air.

Permanent magnets are characterized by their magnetization direction. Since, from a fabrication point of view, continuous magnetization fields are not feasible, we aim to find material distributions with piecewise-constant magnetization directions. More precisely, we allow for ferromagnetic material and $12$ magnetization directions $\varphi_i = \frac{\pi}{6} (i-1)$, $i=1, \dots, 12$, giving a total of 13 candidate materials. Instead of placing these materials in a high-dimensional simplex, we incorporate our domain knowledge and place the magnet materials on a two-dimensional polygon; together with ferromagnetic material placed on top, we get a polytope $P$ with vertices $(\mathrm{cos}\varphi_i, \mathrm{sin}\varphi_i, 0)$ for $i=1,\dots, 12$ and $(0,0,1)$, see Figure \ref{fig:motor}(b).

The problem at hand reads as follows:
\begin{align*}
           & \mathrm{max } \int_\Gamma \mathrm{curl}u \cdot \bm{n} \; \mbox{d}s \qquad \mbox{subject to } \bm{\eta} \in \mathcal K \cap \mathcal M \mbox{ and }                                                              \\
  u \in V: & \int_\Omega \nu(\bm{\eta}, |\mathrm{curl}u|) \mathrm{curl}u \cdot \mathrm{curl}v \; \mbox dx = \int_{\Omega_c} j v \; \mbox{d}x + \int_\Omega M(\bm{\eta}) \cdot \mathrm{curl} v \; \mbox{d}x\; \forall v \in V
\end{align*}
where $u$ denotes the third component of the magnetic vector potential for the magnetic flux density $\mathbf B$ such that $\mathbf B = \mathrm{Curl}((0,0,u)^\top)$, and $\mathrm{curl}(u) = (\partial_2 u, -\partial_1 u)^\top$ denotes the two-dimensional scalar-to-vector $\mathrm{curl}$ operator.
The curve $\Gamma$ lies in the air gap between the inner part (rotor) and the outer part (stator) of the motor, and $\bm{n}$ denotes the outer unit normal vector to $\Gamma$ pointing out of the rotor. The function space $V := \{v \in H^1(\Omega): v = 0 \mbox{ on } \Gamma_{\textrm{in}} \cup \Gamma_{\textrm{out}}, v|_{\Gamma_{\ell}} = - v|_{\Gamma_\textrm{r}} \}$ denotes the space of all $H^1$ functions that vanish on the inner and outer boundaries and that are antiperiodic between $\Gamma_{\ell}$ and $\Gamma_\textrm{r}$. The magnetic reluctivity $\nu$ is interpolated between a constant value $\nu_{\textrm{mag}} = \nu_0/1.03$ for all magnets where $\nu_0 = 10^7 / (4 \pi)$ is the vacuum reluctivity and a nonlinear function $\hat \nu(s)$ obtained from measurement data \cite{Pechstein2006}.
The source current density $j$ is set to a constant of $10$ A/mm$^2$ in $\Omega_{c}^+$ and $-10$ A/mm$^2$ in $\Omega_{c}^-$, see Fig. \ref{fig:motor}(a).
For the interpolation we used
\begin{align*}
  \nu(\bm \eta, s) = \nu_{\textrm{mag}} + \eta_3^{p_{\nu}} (\hat \nu(s) - \nu_{\textrm{mag}}) \quad \mbox{and}\quad
  M(\bm \eta) = \nu_{\textrm{mag}} \sqrt{\eta_1^2+\eta_2^2}^{p_{\textrm{mag}}-1} (\eta_1, \eta_2)^\top
\end{align*}
with $p_\nu=3$ and $p_{\textrm{mag}}=2$.
Finally, $\mathcal K$ is given in \eqref{eq:admin_a} with $P$ as depicted in Figure \ref{fig:motor}(b) and $\mathcal M$ is given in \eqref{eq:admin_b} with $W=[ 0 \; 0\; -1]$ and $b=-0.8 |\Omega_d|$, thus yielding an upper bound of the magnet volume ratio of 20\%.

We discretized the computational domain depicted in Fig. \ref{fig:motor}(a) by a triangular mesh consisting of 18232 vertices and 36151 triangles with a higher resolution in the design domain. We used the standard PDE-based density filter in \eqref{eq:pde_filter} and a bisection method to enforce the volume constraint at each iteration. We again stopped the algorithm when the residual \eqref{eq:KKT_estimator} was reduced by a factor of $10^{-4}$ yielding an absolute residual $\texttt{res}_k= 9.36\times 10^{-8}$. The result of the optimization obtained after 30 iterations is depicted in Figure \ref{fig:motor}(c). The objective value was improved from around $5.35\times 10^{-5}$ for the initial design represented by the latent variable $\bm \psi^0 \equiv (0,1,5)^\top$ to about $1.018\cdot 10^{-2}$ for the optimized design.

\begin{figure}
  \centering
  \begin{tabular}{ccc}
    \includegraphics[width=0.23\linewidth]{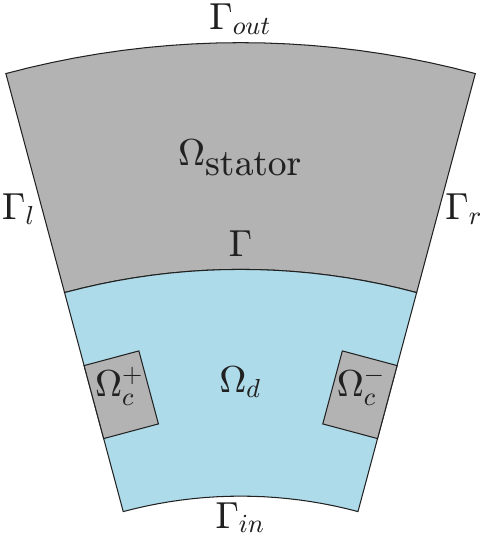}\;\;    &
    \hspace{-5mm}\includegraphics[width=0.35\linewidth]{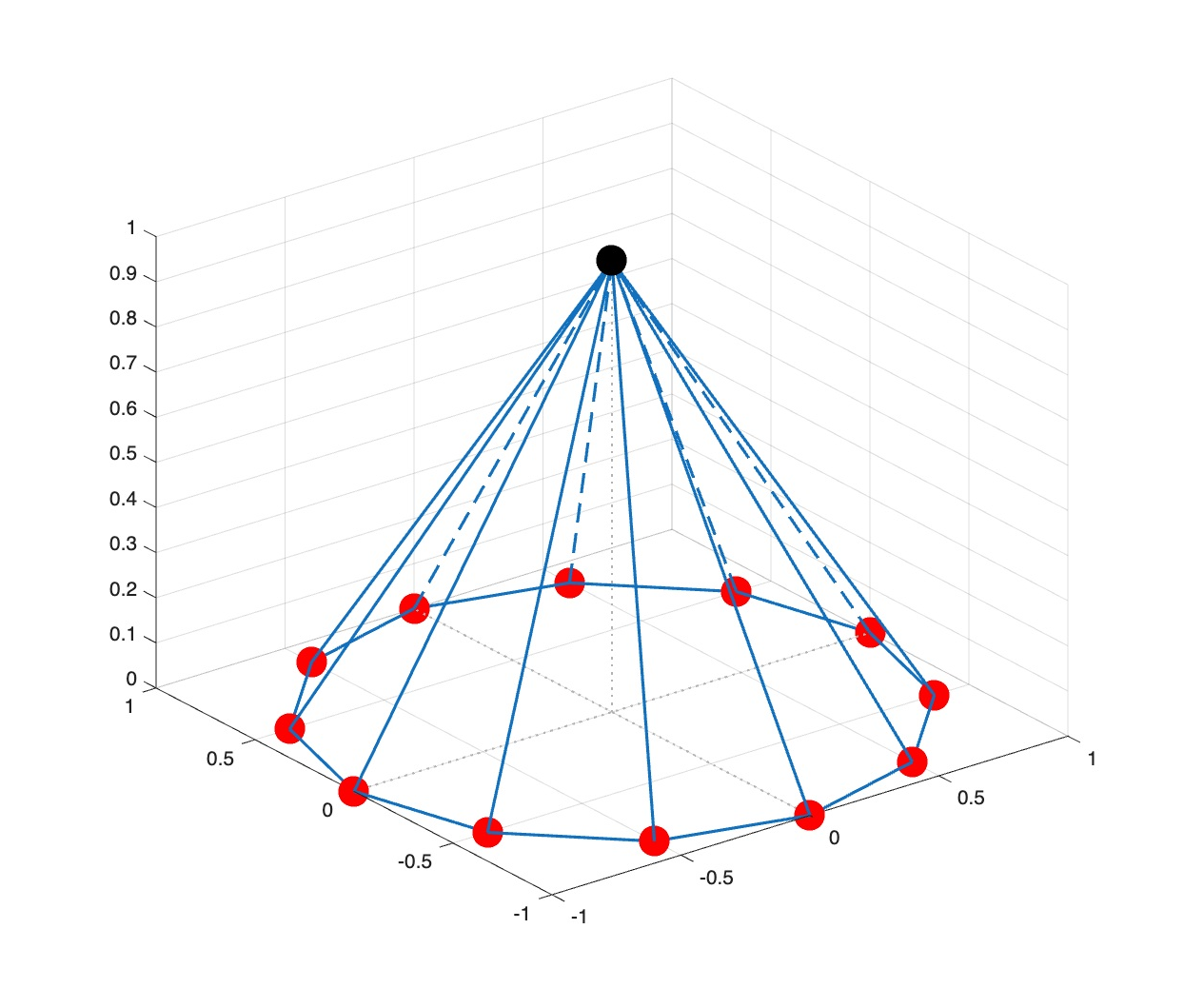} &
    \hspace{-5mm}\includegraphics[width=0.35\linewidth]{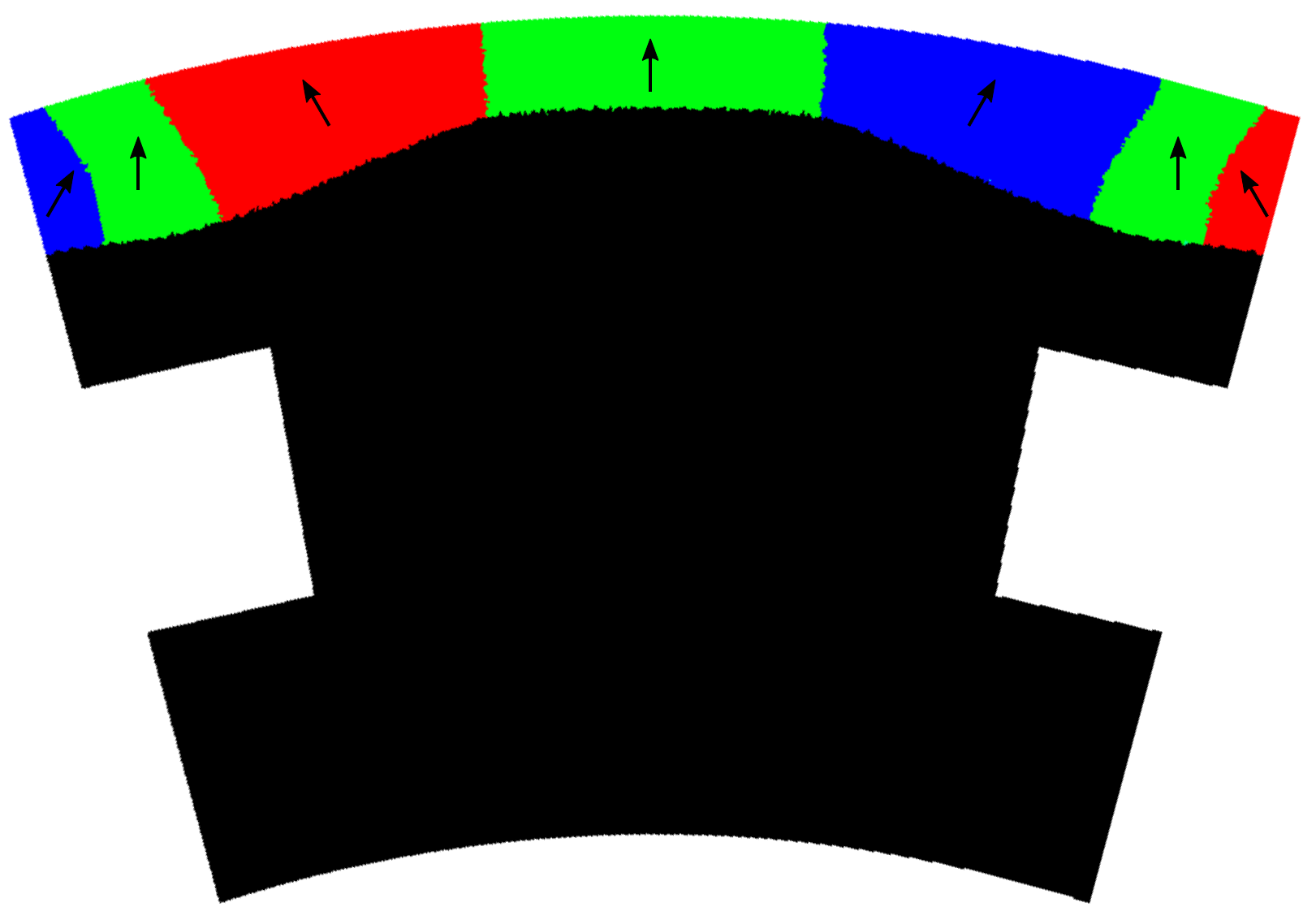}   \\
    (a)                                                                   & (b) & (c)
  \end{tabular}
  \caption{Optimization of electric motor: (a) Cross-section of the sector of an electric motor with the design region $\Omega_d$ highlighted in blue; (b) Convex polytope $P$ for 12 magnetization directions (red) and ferromagnetic material (black); (c) Optimized design consisting of ferromagnetic material (black) and magnetization directions 60$^\circ$ (blue), 90$^\circ$ (green), 120$^\circ$ (red).\label{fig:motor}}
\end{figure}

\section{Conclusion}

We have presented a mirror descent method for multi-material topology optimization over general convex polytopal design domains.
The central contribution is a polytopal entropy function $R$, obtained as the infimal projection of the Gibbs entropy via the vertex matrix $V$, whose conjugate yields the smooth map $\nabla R^*(\bm{\psi}) = V\operatorname{softmax}(V^\top\bm{\psi})$ parametrizing $\operatorname{int}(P)$ by an unconstrained latent variable.
The resulting algorithm decouples into an unconstrained gradient step and a low-dimensional dual update, with iterates that remain provably in $\mathring{\mathcal{K}}$.
Numerical experiments across three problem classes--isotropic compliance minimization, orthotropic orientation optimization, and electric motor rotor design--demonstrate rapid convergence with minimal parameter tuning.
The placement of material phases at the vertices of $P$ allows physical prior knowledge to be encoded directly in the parametrization.

\bibliographystyle{siamplain}
\bibliography{ref}
\end{document}